\documentclass[11pt]{article}
\usepackage{fullpage} 

\usepackage{ucs}
\usepackage{amssymb}
\usepackage{amsthm}
\usepackage{amsmath}
\usepackage{latexsym}
\usepackage[cp1251]{inputenc}
\usepackage[english]{babel}
\usepackage{graphicx}
\usepackage{wrapfig}
\usepackage{caption}
\usepackage{subcaption}
\usepackage{txfonts}
\usepackage{lmodern}
\usepackage{mathrsfs}
\usepackage{algorithm}
\usepackage{dsfont}
\usepackage{enumerate}
\usepackage{multicol}
\usepackage{csquotes}
\usepackage{stmaryrd}
\usepackage{tikz}
\usepackage{comment}

\usepackage{cite}

\usepackage[pdftex,hypertexnames=false,linktocpage=true]{hyperref}
\hypersetup{colorlinks=true,linkcolor=black,anchorcolor=blue,citecolor=black,filecolor=blue,urlcolor=blue,bookmarksnumbered=true,pdfview=FitB}

\newcommand{\vc}[1]{\ensuremath{\vcenter{\hbox{#1}}}}

\newcommand{\oururl}{\url{http://lidicky.name/pub/co2/}}

\DeclareMathOperator{\Img}{Im}
\DeclareMathOperator*{\argmax}{arg\,max}
\newtheorem{theo}{Theorem}
\newtheorem{prop}[theo]{Proposition}
\newtheorem{lemma}[theo]{Lemma}
\newtheorem{ques}[theo]{Question}
\newtheorem{corl}[theo]{Corollary}
\newtheorem{conj}[theo]{Conjecture}
\newtheorem{claim}[theo]{Claim}

\theoremstyle{definition}

\newtheorem{defn}[theo]{Definition}

\numberwithin{theo}{section}

\tikzset{unlabeled_vertex/.style={inner sep=1.7pt, outer sep=0pt, circle, fill}} 
\tikzset{labeled_vertex/.style={inner sep=2.2pt, outer sep=0pt, rectangle, fill=yellow, draw=black}} 
\tikzset{edge_color0/.style={color=black,line width=1.2pt,opacity=0.5}} 
\tikzset{edge_color1/.style={color=red,  line width=1.2pt,opacity=1}} 
\tikzset{edge_color2/.style={color=blue, line width=1.2pt,opacity=1}} 
\tikzset{edge_color3/.style={color=green,line width=1.2pt}} 
\tikzset{edge_color4/.style={color=red,  line width=1.2pt,dotted}} 
\tikzset{edge_color5/.style={color=blue, line width=1.2pt,dotted}} 
\tikzset{edge_color6/.style={color=green, line width=1.2pt,dotted}} 
\tikzset{edge_color7/.style={color=orange, line width=1.2pt}} 
\tikzset{edge_color8/.style={color=gray, line width=1.2pt}} 
\tikzset{edge_thin/.style={color=black}} 
\tikzset{edge_hidden/.style={color=black,dotted,opacity=0}} 
\tikzset{vertex_color1/.style={inner sep=1.7pt, outer sep=0pt, draw, circle, fill=red}} 
\tikzset{vertex_color2/.style={inner sep=1.7pt, outer sep=0pt, draw, circle, fill=blue}} 
\tikzset{vertex_color3/.style={inner sep=1.7pt, outer sep=0pt, draw, circle, fill=green}} 
\tikzset{labeled_vertex_color1/.style={inner sep=2.2pt, outer sep=0pt, draw, rectangle, fill=red}} 
\tikzset{labeled_vertex_color2/.style={inner sep=2.2pt, outer sep=0pt, draw, rectangle, fill=blue}} 
\tikzset{labeled_vertex_color3/.style={inner sep=2.2pt, outer sep=0pt, draw, rectangle, fill=green}}

\tikzset{
vtx/.style={inner sep=1.1pt, outer sep=0pt, circle, fill,draw}, 
vtxl/.style={inner sep=1.1pt, outer sep=0pt, rectangle, fill=yellow,draw=black}, 
hyperedge/.style={fill=pink,opacity=0.5,draw=black}, 
}

\newcommand{\Kfourthree}{\vc{
\begin{tikzpicture}
\draw
(0,0) coordinate(1) node[vtx](a){}
(1,0) coordinate(2) node[vtx](b){}
(1,1) coordinate(3) node[vtx](c){}
(0,1) coordinate(4) node[vtx](d){}
;
\draw[hyperedge] (1) to[out=30,in=90] (2) to[out=100,in=260] (3) to[out=260,in=30] (1);
\draw[hyperedge] (1) to[out=90,in=150] (2) to[out=150,in=280] (4) to[out=280,in=90] (1);
\draw[hyperedge] (1) to[out=60,in=200] (3) to[out=200,in=330] (4) to[out=330,in=60,looseness=1.5] (1);
\draw[hyperedge] (2) to[out=120,in=230,,looseness=1.8] (3) to[out=230,in=300,looseness=1.8] (4) to[out=300,in=120,looseness=1.8] (2);
\end{tikzpicture}
}}

\newcommand{\Kfourthreeminus}{\vc{
\begin{tikzpicture}
\draw
(0,0) coordinate(1) node[vtx](a){}
(1,0) coordinate(2) node[vtx](b){}
(1,1) coordinate(3) node[vtx](c){}
(0,1) coordinate(4) node[vtx](d){}
;
\draw[hyperedge] (1) to[out=30,in=90] (2) to[out=100,in=260] (3) to[out=260,in=30] (1);
\draw[hyperedge] (1) to[out=90,in=150] (2) to[out=150,in=280] (4) to[out=280,in=90] (1);
\draw[hyperedge] (1) to[out=60,in=200] (3) to[out=200,in=330] (4) to[out=330,in=60,looseness=1.5] (1);
\end{tikzpicture}
}}

\title{Solving Tur\'an's Tetrahedron Problem for the $\ell_2$-Norm}

\begin{document}

\author{%
  J\'ozsef Balogh \footnote{Department of Mathematics, University of Illinois at Urbana-Champaign, Urbana, Illinois 61801, USA, and Moscow Institute of Physics and Technology, Russian Federation. E-mail: \texttt{jobal@illinois.edu}. Research is partially supported by NSF Grant DMS-1764123, Arnold O. Beckman Research
Award (UIUC Campus Research Board RB 18132), the Langan Scholar Fund (UIUC), and the Simons Fellowship.}
\and Felix Christian Clemen \footnote {Department of Mathematics, University of Illinois at Urbana-Champaign, Urbana, Illinois 61801, USA, E-mail: \texttt{fclemen2@illinois.edu}.}
 \and Bernard Lidick\'{y} \footnote {Iowa State University, Department of Mathematics, Iowa State University, Ames, IA., E-mail: \texttt{ lidicky@} \texttt{iastate.edu}. Research of this author is partially supported by NSF grant DMS-1855653.}
}
\date{\today}
\maketitle
\abstract{Tur\'an's famous tetrahedron problem is to compute the Tur\'an density of the tetrahedron $K_4^3$. This is equivalent to determining the maximum $\ell_1$-norm of the codegree vector of a $K_4^3$-free $n$-vertex $3$-uniform hypergraph.  
We introduce a new way for measuring extremality of hypergraphs and determine asymptotically the extremal function of the tetrahedron in our notion.

The codegree squared sum, $\textup{co}_2(G)$, of a $3$-uniform hypergraph $G$ is the sum of codegrees squared $d(x,y)^2$ over all pairs of vertices $xy$, or in other words, the square of the $\ell_2$-norm of the codegree vector of the pairs of vertices. We define $\textup{exco}_2(n,H)$ to be the maximum $\textup{co}_2(G)$ over all $H$-free $n$-vertex $3$-uniform hypergraphs $G$. 
We use flag algebra computations to determine asymptotically the codegree squared extremal number for $K_4^3$ and $K_5^3$ and additionally prove stability results. 

In particular, we prove that the extremal 
$K_4^3$-free hypergraphs in $\ell_2$-norm have approximately the same structure as one of the conjectured extremal hypergraphs for Tur\'an's conjecture. Further, we prove several general properties about $\textup{exco}_2(n,H)$ including the existence of a scaled limit, blow-up invariance and a supersaturation result.}
\section{Introduction}
For a $k$-uniform hypergraph $H$ (shortly $k$-graph), the Tur\'an function (or extremal number) $\textup{ex}(n,H)$ is the maximum number of edges in an $H$-free $n$-vertex $k$-uniform hypergraph. The graph case, $k=2$, is reasonably well-understood. The classical Erd\H{o}s-Stone-Simonovits theorem \cite{ErdosStone,ErdosSimonovits} determines asymptotically the extremal number for graphs with chromatic number at least three. However, for general $k$, the problem of determining the extremal function is much harder and widely open. Despite enormous efforts, our understanding of Tur\'an functions is still limited. Even the extremal function of the \textit{tetrahedron} $K_4^3$, the $3$-graph on $4$ vertices with $4$ edges, is unknown. There are exponentially (in the number of vertices) many conjectured extremal hypergraphs which is believed to be the root of the difficulty of this problem. Brown~\cite{K43brown}, Kostochka~\cite{K43Kostochka}, Fon-der-Flaass~\cite{K43Fonderflaass} and Frohmader~\cite{MR2465761} constructed families of $K_4^3$-free $3$-graphs which they conjectured to be extremal. For an excellent survey on Tur\'an functions of cliques see \cite{MR1341481} by Sidorenko. 

Successively, the upper bound for extremal number of the tetrahedron has been improved by de Caen~\cite{TetrahedronCaen},  Giraud (unpublished, see \cite{ChungLutetrahedron}), Chung and Lu~\cite{ChungLutetrahedron}, and finally Razborov~\cite{RazbarovK43} and Baber~\cite{Baber}, both making use of Razborov's flag algebra approach~\cite{flagsRaz} (see also Baber and Talbot~\cite{BaberTalbot}). Another relevant result towards solving Tur\'an's tetrahedron problem is by Pikhurko~\cite{PikhurkoK43}. Building on a result by Razborov~\cite{RazbarovK43}, Pikhurko~\cite{PikhurkoK43} determined the exact ext\-re\-mal hypergraph when the induced $4$-vertex graph with one edge is forbidden in addition to the tetrahedron.  \par
In this paper we study a different notion of extremality and solve the tetrahedron problem asymptotically for this notion. It is interesting that the extremal 
 $K_4^3$-free hypergraphs in $\ell_2$-norm have approximately the same structure as one of the conjectured extremal hypergraphs for Tur\'an's conjecture. For an integer $n$, denote by $[n]$ the set of the first $n$ integers. Given a set $A$ and an integer $k$, we write $\binom{A}{k}$ for the set of all subsets of $A$ of size $k$. Let $G$ be an $n$-vertex $k$-uniform hypergraph. For $T\subset V(G)$ with $|T|=k-1$ we denote by $d_G(T)$ the \emph{codegree} of $T$, i.e., the number of edges in $G$ containing $T$. If the choice of $G$ is obvious, we will drop the index and just write $d(T)$. The \emph{codegree vector} of $G$ is the vector 
\begin{align*}
X \in \mathbb{Z}^{\binom{V(G)}{k-1}}, \text{ where } X(v_1,v_2,\ldots,v_{k-1}) = d(v_1,v_2,\ldots,v_{k-1})
\end{align*}
for every $\{v_1,v_2,\ldots,v_{k-1}\} \in \binom{V(G)}{k-1}$. The $\ell_1$-norm of the codegree vector, or to put it in other words, the sum of codegrees, is $k$ times the number of edges. Thus, Tur\'an's problem for $k$-graphs is equivalent to the question of finding the maximum $\ell_1$-norm for the codegree vector of $H$-free $k$-graphs. We propose to study this maximum with respect to other norms. A particular interesting case seems to be the $\ell_2$-norm of the codegree vector. We will refer to the square of the $\ell_2$-norm of the codegree vector as the \emph{codegree squared sum} denoted by $\textup{co}_2(G),$
\begin{align*}
\textup{co}_2(G)=\sum_{\substack{T \subset \binom{[n]}{k-1}\\ |T|=k-1} }d_G^2(T). \end{align*}
\begin{ques}
\label{questionco2}
Given a $k$-uniform hypergraph $H$, what is the maximum codegree squared sum a $k$-uniform $H$-free $n$-vertex hypergraph $G$ can have?
\end{ques}

Many different types of extremality in hypergraphs have been studied: 

The most related one is the minimum codegree-threshold. For a given $k$-graph, the \emph{minimum codegree-threshold} is the largest minimum codegree an $n$-vertex $k$-graph can have without containing a copy of $H$. This problem has not even been solved for $H$ being the tetrahedron. For a collection of results on the minimum codegree-threshold see \cite{codegreeconj, LoMark, Mubayicode, codF32falgas,codegreeFalgas, FalgasK4-, Sidorenkocode, MubayiFano, AllanZhao}. 

Reiher, R\"{o}dl and Schacht~\cite{MR3764068,MR3548293} introduced new variants of the Tur\'an density, which ask for the maximum density for which an $H$-free hypergraph with a certain quasirandomness property exists. Roughly speaking, a quasirandomness property is a property which holds for the random hypergraph with high probability. Reiher, R\"{o}dl and Schacht~\cite{MR3548293} determined such a variant for the tetrahedron.

In this paper we solve asymptotically Question~\ref{questionco2} for the tetrahedron. For a family $\mathcal{F}$ of $k$-uniform hypergraphs, we define $\textup{exco}_2(n,\mathcal{F})$ to be the maximum codegree squared sum a $k$-uniform $n$-vertex $\mathcal{F}$-free hypergraph can have, and the \emph{codegree squared density} $\sigma(F)$ to be its scaled limit, i.e.,
\begin{align}
\label{sigmalimit}
    \textup{exco}_2(n,\mathcal{F})= \displaystyle\max_{\substack{G \text{ is an } n \text{-vertex }\\ \mathcal{F} \text{-free} \\ k \text{-uniform hypergraph}}} \textup{co}_2(G) \quad \quad \text{ and } \quad \quad
   \sigma(\mathcal{F})&=\lim\limits_{n \rightarrow \infty} \frac{\textup{exco}_2(n,\mathcal{F})}{\binom{n}{k-1}(n-k+1)^2}.
\end{align}
We will observe in Proposition~\ref{coslim} that the limit in \eqref{sigmalimit} exists. Denote by $K_\ell^3$ the complete $3$-uniform hypergraph on $\ell$ vertices. Our main result is that we determine the codegree squared density asymptotically for $K_4^3$ and $K_5^3$, respectively. 
\begin{theo}
\label{K43K53}
We have 
\begin{align*}
    \sigma(K_4^3)=\frac{1}{3} \quad \quad \text{and} \quad \quad \sigma(K_5^3)=\frac{5}{8}.
\end{align*}
\end{theo}
Denote by $C_n$ the $3$-uniform hypergraph\footnote{This hypergraph is often referred to as Tur\'an's construction.} on $n$ vertices with vertex set $V(C_n)=V_1\cup V_2 \cup V_3$ such that $||V_i|-|V_j||\leq 1$ for $i\neq j$ and edge set 
\begin{align*}
    E(C_n)=\{abc: a\in V_1,b\in V_2, c\in V_3 \} \cup \{abc: a,b\in V_1,c\in V_2 \} \\ \cup \ \{abc: a,b\in V_2,c\in V_3 \} \cup  \{abc: a,b\in V_3,c\in V_1 \}.  
\end{align*}
Further, denote by $B_n$ the balanced, complete, bipartite $3$-uniform hypergraph on $n$ vertices, that is the hypergraph where the vertex set is partitioned into two sets $A,B$ such that $||A|-|B||\leq 1$ and the edge set is the set of triples intersecting both $A$ and $B$. See Figure~\ref{fig:CnBn} for an illustration of $C_n$ and $B_n$.
The $3$-graphs $C_n$ and $B_n$ are among the asymptotically extremal hypergraphs in $\ell_1$-norm for $K_4^3$ and $K_5^3$ respectively. We conjecture that $C_n$ and $B_n$ are the unique extremal hypergraphs in $\ell_2$-norm.

\begin{figure}
\begin{center}
\tikzset{
vtx/.style={inner sep=1.1pt, outer sep=0pt, circle, fill,draw}, 
hyperedge/.style={fill=pink,opacity=0.5,draw=black}, 
vtxBig/.style={inner sep=12pt, outer sep=0pt, circle, fill=white,draw}, 
hyperedge/.style={fill=pink,opacity=0.5,draw=black}, 
}
\vc{
\begin{tikzpicture}[scale=1.4]
\draw (30:0.9) coordinate(x1) node[vtxBig]{};
\draw (150:0.9) coordinate(x2) node[vtxBig]{};
\draw (270:0.9) coordinate(x3) node[vtxBig]{};
\draw
(30:0.8) coordinate(1) 
(150:0.8) coordinate(2) 
(270:0.8) coordinate(3) 
;
\draw[hyperedge] (1) to[out=210,in=330] (2) to[out=330,in=90] (3) to[out=90,in=210] (1);
\foreach \ashift in {30,150,270}{
\draw
(\ashift:0.8)++(\ashift+60:0.1) coordinate(1) 
++(\ashift+60:0.3) coordinate(2) 
(\ashift+120:0.8)++(\ashift+120+90:-0.2) coordinate(3)
;
\draw[hyperedge] (1) to[bend left] (2) to[bend left=5] (3) to[bend left=5] (1);
}
\draw (30:1)  node{$V_1$};
\draw (150:1) node{$V_2$};
\draw (270:1) node{$V_3$};
\draw (0,-1.8) node{$C_n$};
\end{tikzpicture}
}
\hskip 2cm
\vc{
\begin{tikzpicture}[scale=2.0]
\draw (0,0) coordinate(x1) ellipse (0.25cm and 0.7cm);
\draw (1,0) coordinate(x1) ellipse (0.25cm and 0.7cm);
\draw
(0,-0.2) coordinate(1) 
(0,-0.4) coordinate(2) 
(0,0.3) coordinate(3) 
(1,0.2) coordinate(4) 
(1,-0.3) coordinate(5) 
(1,0.4) coordinate(6) 
;
\draw[hyperedge] (1) to[bend left] (2) to[bend left=5] (5) to[bend left=5] (1);
\draw[hyperedge] (4) to[bend left] (6) to[bend left=5] (3) to[bend left=5] (4);
\draw (0,0) node{$A$};
\draw (1,-0) node{$B$};
\draw (0.5,-1) node{$B_n$};
\end{tikzpicture}
}
\end{center}
\caption{Illustration of $C_n$ and $B_n$.}\label{fig:CnBn}
\end{figure}
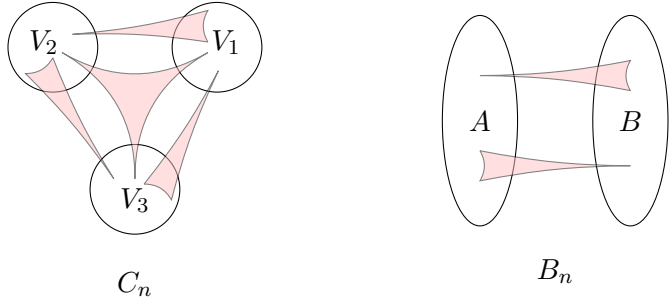

\begin{conj}
  There exists $n_0$ such that for all $n\geq n_0$
\begin{align*}
    \textup{exco}_2(n,K_4^3)= \textup{co}_2(C_n),
\end{align*}
and $C_n$ is the unique $K_4^3$-free $n$-vertex $3$-uniform hypergraph with codegree squared sum equal to $\textup{exco}_2(n,K_4^3)$. 
\end{conj}
Note that Kostochka's~\cite{K43Kostochka} result suggests that in the $\ell_1$-norm there are exponentially many extremal graphs, $C_n$ is one of them. 
\begin{conj}
  There exists $n_0$ such that for all $n\geq n_0$
\begin{align*}
    \textup{exco}_2(n,K_5^3)= \textup{co}_2(B_n),
\end{align*}
and $B_n$ is the unique $K_5^3$-free $n$-vertex $3$-uniform hypergraph with codegree squared sum equal to $\textup{exco}_2(n,K_5^3)$. 
\end{conj}
We believe that existing methods could prove these conjectures, though the potential proofs might be long and technical. 

In Section~\ref{sec:flaga} we observe that giving upper bounds on $\sigma(H)$ for some $3$-graph $H$ is equivalent to giving upper bounds on a certain linear combination of densities of $4$-vertex subgraphs in large $H$-free graphs, see \eqref{eq:fa}. By now it is a standard technique in the field to use the computer-assisted method of flag algebras to prove such bounds. If one gets an asymptotically tight upper bound from a flag algebra computation, it is typically the case that there is an essentially unique stable extremal example and that one can extract a stability result from the flag algebra proof. This also happens for $K_4^3$ and $K_5^3$. For $\varepsilon>0$, we say a given $n$-vertex $3$-graph $H$ is \textit{$\varepsilon$-near} to an $n$-vertex $3$-graph $G$ if there exists a bijection $\phi : V(G) \rightarrow V(H)$ such that the number of $3$-sets $xyz$ satisfying $xyz\in E(G),\phi(x)\phi(y)\phi(z)\not\in E(H)$ or $xyz\not\in E(G),\phi(x)\phi(y)\phi(z)\in E(H)$ is 
at most $\varepsilon |V(H)|^3$.

\begin{theo}
\label{K43stability}
For every $\varepsilon>0$ there exists $\delta>0$ and $n_0$ such that for every $n > n_0$, if $G$ is a $K_4^3$-free $3$-uniform hypergraph on $n$ vertices with 
\begin{align*}
\textup{co}_2(G)\geq \left(\frac{1}{3}-\delta \right) \frac{n^4}{2}, 
\end{align*}
then $G$ is $\varepsilon$-near to $C_n$.
\end{theo}

\begin{theo}
\label{K53stability}
For every $\varepsilon>0$ there exists $\delta>0$ and $n_0$ such that for every $n > n_0$, if $G$ is a $K_5^3$-free $3$-uniform hypergraph on $n$ vertices with 
\begin{align*}
\textup{co}_2(G)\geq \left(\frac{5}{8}-\delta \right) \frac{n^4}{2}, 
\end{align*}
then $G$ is $\varepsilon$-near to $B_n$.
\end{theo}
There is another $K_5^3$-free $3$-graph~\cite{Sidorenko1981SystemsOS} with the same edge density as $B_n$, namely $H_5$. The vertex set of $H_5$ is divided into $4$ parts $A_1, A_2, A_3, A_4$ with $||A_j|-|A_i||\leq 1$ for all $1\leq i\leq j\leq 4$ and say a triple $e$ is not an edge of $H_{5}$ iff there is some $j$ ($1\leq j \leq 4$) such that $|e \cap A_{j}|\geq 2$ and $|e\cap A_{j}|+|e\cap A_{j+1}|=3$, where $A_{5}=A_1$, see Figure~\ref{fig:H5F33} for an illustration of the complement of $H_5$. While $H_5$ is conjectured to be one of the asymptotically extremal hypergraphs in $\ell_1$-norm, it is not an extremal hypergraph in $\ell_2$-norm, because $B_n$ has an asymptotically higher codegree squared sum.

Besides giving asymptotic result for cliques, we prove an exact result for $F_{3,3}$. Denote by $F_{3,3}$ the $3$-graph on $6$ vertices with edge set $\{ 123, 145, 146, 156, 245, 246, 256, 345,$ $346, 356 \}$, see Figure~\ref{fig:H5F33}. We prove that the codegree squared extremal hypergraph of $F_{3,3}$ is the balanced, complete, bipartite hypergraph $B_n$. Keevash and Mubayi~\cite{MR2912791} and independently Goldwasser and Hansen~\cite{MR3048207} proved that $B_n$ is also extremal for the $\ell_1$-norm.

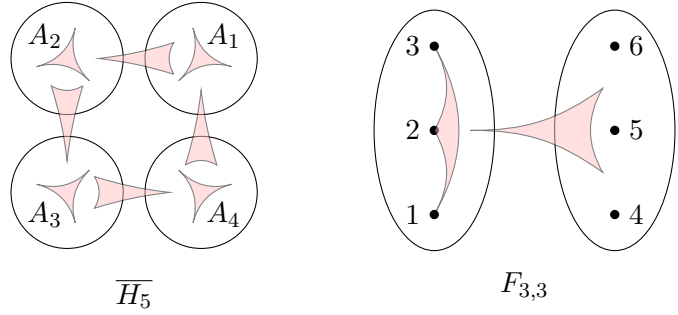
\begin{figure}
    \begin{center}
\tikzset{
vtx/.style={inner sep=1.1pt, outer sep=0pt, circle, fill,draw}, 
hyperedge/.style={fill=pink,opacity=0.5,draw=black}, 
vtxBig/.style={inner sep=15pt, outer sep=0pt, circle, fill=white,draw}, 
hyperedge/.style={fill=pink,opacity=0.5,draw=black}, 
}
\vc{
\begin{tikzpicture}[scale=1.4]
\draw (45:0.9) coordinate(x1) node[vtxBig]{};
\draw (135:0.9) coordinate(x2) node[vtxBig]{};
\draw (225:0.9) coordinate(x3) node[vtxBig]{};
\draw (315:0.9) coordinate(x4) node[vtxBig]{};
\foreach \ashift in {45,135,225,315}{
\draw
(\ashift:0.9)
+(\ashift+135-30:0.3) coordinate(2) 
+(\ashift+135+30:0.3) coordinate(1) 
(\ashift+90:0.9)++(\ashift-45:0.3) coordinate(3)
;
\draw[hyperedge] (1) to[bend left] (2) to[bend left=5] (3) to[bend left=5] (1);
\draw
(\ashift:0.9)
+(\ashift+180:0.3) coordinate(1) 
+(\ashift+180+120:0.3) coordinate(2) 
+(\ashift+180+240:0.3) coordinate(3) 
;
\draw[hyperedge] (1) to[bend left] (2) to[bend left] (3) to[bend left] (1);
}
\draw (45:1.2)  node{$A_1$};
\draw (135:1.2) node{$A_2$};
\draw (225:1.2) node{$A_3$};
\draw (315:1.2) node{$A_4$};
\draw (0,-1.6) node{$\overline{H_5}$};
\end{tikzpicture}
}
        \hskip 3em
\vc{
\begin{tikzpicture}[scale=1.6]
\draw
(0,-0.7) coordinate(1) node[vtx,label=left:1](a){}
(0,0) coordinate(2) node[vtx,label=left:2](b){}
(0,0.7) coordinate(3) node[vtx,label=left:3](c){}
(1.5,-0.7) coordinate(4) node[vtx,label=right:4](d){}
(1.5,0.0) coordinate(5) node[vtx,label=right:5](d){}
(1.5,0.7) coordinate(6) node[vtx,label=right:6](d){}
;
\draw[hyperedge] (1) to[bend right] (2) to[bend right] (3) to[bend left] (1);
\draw
(0.3,0) coordinate(1)
(1.5-0.1,0.35) coordinate(2)
(1.5-0.1,-0.35) coordinate(3)
(0,0) ellipse (0.5 cm and 1 cm)
(1.5,0) ellipse (0.5 cm and 1 cm)
;
\draw[hyperedge] (1) to[bend right=18] (2) to[bend right] (3) to[bend right=18] (1);
\draw (0.75,-1.3) node {$F_{3,3}$};
\end{tikzpicture}
}
    \end{center}
    \caption{Left: The complement of $H_5$. Right: A sketch of $F_{3,3}$, which has 6 vertices and edge set $\{ 123, 145, 146, 156, 245, 246, 256, 345,$ $346, 356 \}$.}
    \label{fig:H5F33}
\end{figure}

\begin{theo}
\label{F33exact}
There exists $n_0$ such that for all $n\geq n_0$
\begin{align*}
 \textup{exco}_2(n,F_{3,3})= \textup{co}_2(B_n). 
\end{align*}
Furthermore, $B_n$ is the unique $F_{3,3}$-free $3$-uniform hypergraph $G$ on $n$ vertices satisfying \begin{align*}
\textup{co}_2(G)   =\textup{exco}_2(n,F_{3,3}). 
\end{align*}
\end{theo}

We also prove some general results for $\sigma$. First, we prove that the limit in \eqref{sigmalimit} exists.
\begin{prop}
\label{coslim}
Let $\mathcal{F}$ be a family of $k$-graphs. Then,
$ \frac{\textup{exco}_2(n,\mathcal{F})}{\binom{n}{k-1}(n-k+1)^2}$ is non-increasing as $n$ increases. In particular, it tends to a limit $\sigma(\mathcal{F})$ as $n\rightarrow \infty$.
\end{prop}
A classical result in extremal combinatorics is the supersaturation phenomenon, discovered by Erd\H{o}s and Simonovits \cite{Supersaterdos}. For hypergraphs it states, that when the edge density of a hypergraph $H$ exceeds the Tur\'an density of a hypergraph $G$, then $H$ contains many copies of $G$. Proposition~{\ref{cossupersat}} shows that the same phenomenon holds for $\sigma$.

\begin{prop}
\label{cossupersat}
Let $F$ be a $k$-graph on $f$ vertices. For every $\varepsilon>0$, there exists $\delta=\delta(\varepsilon,f)>0$ and $n_0$ such that every $n$-vertex $k$-uniform hypergraph $G$ with $n>n_0$ and $\textup{co}_2(G)>(\sigma(F)+\varepsilon) \binom{n}{k-1}n^2$ contains at least  $\delta\binom{n}{f}$ copies of $F$. 
\end{prop}
Supersaturation has been used to show that blowing-up a $k$-graph does not change its Tur\'an density \cite{Supersaterdos}. We will use our Supersaturation result, Proposition~\ref{cossupersat}, to show the same conclusion holds for $\sigma$: Blowing-up a $k$-graph also does not change the codegree squared density.\\
For a $k$-graph $H$ and $t\in \mathbb{N}$, the \emph{blow-up} $H(t)$ of $H$ is defined by replacing each vertex $x\in V(H)$ by $t$ vertices $x^1,\ldots,x^t$ and each edge $x_1 \cdots x_k\in E(H)$ by the $t^k$ edges $x_1^{a_1}\cdots x_k^{a_k}$ with $1 \leq  a_1,\ldots, a_k \leq t$.  

\begin{corl}
\label{blowup}
Let $H$ be a $k$-uniform hypergraph and $t\in \mathbb{N}$. Then,
$$ \sigma(H)=\sigma(H(t)).$$
\end{corl}
Similarly to the Tur\'an density~\cite{Erdos3partite}, the codegree squared density has a jump at $0$, i.e. it is strictly bounded away from $0$. Note that this phenomenon does not happen for the minimum codegree threshold~\cite{LoMark}.
\begin{prop}
\label{jumps}
Let $H$ be a $k$-uniform hypergraph. Then
\begin{itemize}
    \item[(i)] $(\pi(H))^2\leq \sigma(H)\leq \pi(H)$,
    \item[(ii)] $\sigma(H)=0$ or $\sigma(H)\geq \frac{(k-1)!}{k^k}$. 
\end{itemize} 
\end{prop}

Our paper is organised as follows. In Section~\ref{cancellative}, as a warm up, we determine the maximum $\ell_2$-norm of cancellative\footnote{A hypergraph is called \emph{cancellative} if it is $\{F_4,F_5\}$-free. See Section~\ref{cancellative} for the definition of $F_4$ and $F_5$.} $3$-graphs, which is an analogue of a classical result of Bollob\'as~\cite{Bollobascancellative}. 
Next, in Section~\ref{preperation} we introduce terminology and give an overview of the tools we will be using. In Section~\ref{General results} we present our general results on maximal codegree squared sums. Section~\ref{cliques} is dedicated to proving our main results on cliques, i.e., proving Theorems~\ref{K43stability} and \ref{K53stability}. In Section~\ref{F33section} we present the proof of our exact result, Theorem~\ref{F33exact}. 

In a follow-up paper~\cite{BalCleLid}, we systematically study the codegree squared densities of several hypergraphs, including a longer discussion of related open problems.

\section{\texorpdfstring{Forbidding  $F_4$ and $F_5$}{TEXT}}
\label{cancellative}
In this section we will provide an example of how a classical Tur\'an-type result on the $\ell_1$-norm can imply a result for the $\ell_2$-norm. 
Denote by $F_4$ the $4$-vertex $3$-graph\footnote{This hypergraph is also knows as $K_4^{3-}$.} with edge set $\{123,124,234\}$ and $F_5$ the $5$-vertex $3$-graph with edge set $\{123,124,345\}$, see Figure~\ref{fig:f4f5}. The $3$-graphs which are $F_4$- and $F_5$-free are called \emph{cancellative hypergraphs}. Denote by $S_n$ the complete balanced $3$-partite $3$-graph on $n$ vertices. This is the $3$-graph with vertex partition $A\cup B \cup C$ with part sizes $|A|=\lfloor n/3\rfloor$, $|B|=\lfloor (n+1)/3\rfloor$ and $|C|=\lfloor (n+2)/3\rfloor$, where triples $abc$ are edges iff $a,b$ and $c$ are each from a different class. Bollob\'as~\cite{Bollobascancellative} proved that the $n$-vertex cancellative hypergraph with the most edges is $S_n$. Using his result and a double counting argument we show that $S_n$ is also the largest cancellative hypergraph in the $\ell_2$-norm. 

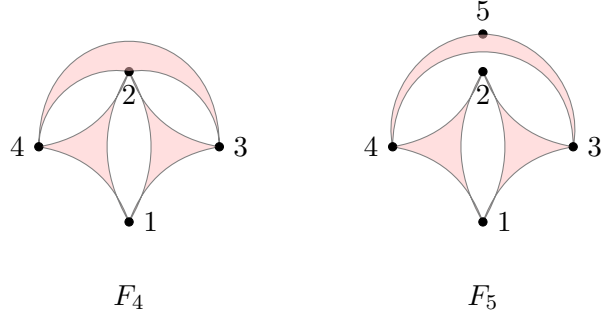
\begin{figure}
    \begin{center}
\tikzset{
vtx/.style={inner sep=1.1pt, outer sep=0pt, circle, fill,draw}, 
hyperedge/.style={fill=pink,opacity=0.5,draw=black}, 
hyperedgeg/.style={fill=gray,opacity=0.5,draw=black}, 
vtxBig/.style={inner sep=17pt, outer sep=0pt, circle, fill=white,draw}, 
hyperedge/.style={fill=pink,opacity=0.5,draw=black}, 
}    
\begin{tikzpicture}[scale=1]
\draw
(0,0) coordinate(1) node[vtx,label=right:1](a){}
(0,2) coordinate(2) node[vtx,label=below:2](b){}
(1.2,1) coordinate(3) node[vtx,label=right:3](c){}
(-1.2,1) coordinate(4) node[vtx,label=left:4](c){}
;
\draw[hyperedge] (1) to[bend right] (2) to[bend right] (3) to[bend right] (1);
\draw[hyperedge] (1) to[bend left] (2) to[bend left] (4) to[bend left] (1);
\draw[hyperedge] (3) to[bend right=50] (2) to[bend right=50] (4) to[bend left=90,looseness=2] (3);
\draw (0,-1) node {$F_4$};
\end{tikzpicture}
\hskip 3em
\begin{tikzpicture}[scale=1]
\draw
(0,0) coordinate(1) node[vtx,label=right:1](a){}
(0,2) coordinate(2) node[vtx,label=below:2](b){}
(1.2,1) coordinate(3) node[vtx,label=right:3](c){}
(-1.2,1) coordinate(4) node[vtx,label=left:4](c){}
(0,2.5) coordinate(5) node[vtx,label=above:5](b){}
;
\draw[hyperedge] (1) to[bend right] (2) to[bend right] (3) to[bend right] (1);
\draw[hyperedge] (1) to[bend left] (2) to[bend left] (4) to[bend left] (1);
\draw[hyperedge] (3) to[bend right=50] (5) to[bend right=50] (4) to[bend left=90,looseness=1.8] (3);
\draw (0,-1) node {$F_5$};
\end{tikzpicture}
    \end{center}
    \caption{The hypergraphs $F_4$ and $F_5$.}
    \label{fig:f4f5}
\end{figure}

\begin{theo}
\label{F4F5}Let $n\in \mathbb{N}$. We have
\begin{align*} 
\textup{exco}_2(n,\{F_4,F_5\})= \textup{co}_2(S_n),
\end{align*}
and therefore also 
$$\sigma(\{F_4,F_5\})=\frac{2}{27}.$$
The unique extremal hypergraph is $S_n$. 
\end{theo}
\begin{proof}
Let $G$ be an $F_4$- and $F_5$-free hypergraph with $n$ vertices.
For an edge $e=xyz\in E(G)$, we define its weight $w(e)=d(x,y)+d(x,z)+d(y,z)$. Then, $w(e)\leq n$; otherwise $G$ contains an $F_4$. 
Bollob\'as \cite{Bollobascancellative} proved that $|E(G)|\leq |E(S_n)|$ with equality iff $G=S_n$. This allows us to conclude
\begin{equation*}
\textup{co}_2(G)= \sum_{xy\in \binom{[n]}{2}}d(x,y)^2= \sum_{e\in E(G)}w(e) \leq n |E(G)|\leq n|E(S_n)|=\textup{co}_2(S_n).
\qedhere
\end{equation*}
\end{proof}
 Frankl and F\"uredi~\cite{F5Frankl} proved that for $F_5$-free $3$-graphs, $S_n$ is also the extremal example in the $\ell_1$-norm when $n\geq 3000$. 
 In a follow-up paper~\cite{BalCleLid} we prove that for $F_5$-free $3$-graphs, $S_n$ is also the extremal example in the $\ell_2$-norm provided $n$ is sufficiently large. However, this requires more work than the proof of Theorem~\ref{F4F5} and it is not derived by just applying the corresponding Tur\'an result.  

\section{Preliminaries}
\label{preperation}
\subsection{Terminology and notation}
Let $H$ be a $3$-uniform hypergraph, $x\in V(H)$ and $A,B\subseteq V(H)$ be disjoint sets. 
\begin{enumerate}
    \item $L(x)$ denotes the link graph of $x$, i.e., the graph on $V(H)\setminus \{x\}$ with $ab\in E(L(x))$ iff $abx\in E(H)$. 
    \item $L_A(x)=L(x)[A]$ denotes the induced link graph on $A$. 
    \item $L_{A,B}(x)$ denotes the subgraph of the link graph of $x$ containing only edges between $A$ and $B$. This means $V(L_{A,B}(x))=V(H)\setminus\{x\}$ and $ab\in E(L_{A,B}(x))$ iff $a\in A, b\in B$ and $abx\in E(H)$.
    \item $L^c_{A,B}(x)$ denotes the subgraph of the link graph of $x$ containing only non-edges between $A$ and $B$. This means $V(L_{A,B}(x))=V(H)\setminus\{x\}$ and $ab\in E(L^c_{A,B}(x))$ iff $a\in A, b\in B$ and $abx \not\in E(H)$.
    \item $e(A,B)$ denotes the number of cross-edges between $A$ and $B$, this means  $$e(A,B):=|\{xyz\in E(H): x,y\in A, z\in B \}|+ |\{xyz\in E(H): x,y\in B, z\in A \}|.$$
    \item $e^c(A,B)$ denotes the number of missing cross-edges between $A$ and $B$, this means  $$e^c(A,B):=\binom{|A|}{2}|B|+ \binom{|B|}{2}|A|-e(A,B).$$
    \item For an edge $e=xyz\in E(H)$, we define its \emph{weight} as \[w_H(e)=d(x,y)+d(x,z)+d(y,z).\] 
\end{enumerate}

\subsection{Tool 1: Induced hypergraph removal Lemma}
We will use the induced hypergraph removal lemma of R\"odl and Schacht \cite{indremoval}.

\begin{defn}
Let $\mathcal{F}$ be an arbitrary family of $k$-graphs and $\mathcal{P}$ be a familiy of $k$-graphs closed under relabeling of the vertices.
\begin{itemize}
    \item $\text{Forb}_{ind}(\mathcal{F})$ denotes the family of all $k$-graphs $H$ which contain no induced copy of any member of $\mathcal{F}$.
    \item For a constant $\mu \geq 0$ we say a given $k$-graph $H$ is $\mu$-\emph{far} from $\mathcal{P}$ if every $k$-graph $G$ on the same vertex set $V(H)$ with $|G \triangle H |\leq \mu |V(H)|^k$ satisfies $G\not \in \mathcal{P}$, where $G \triangle H$ denotes the symmetric difference of the edge sets of $G$ and $H$. Otherwise we call $H$ $\mu$-\emph{near} to $\mathcal{P}$.  
\end{itemize}
\end{defn}

\begin{theo}[R\"odl, Schacht \cite{indremoval}]
\label{removallemma}
For every (possibly infinite) family $\mathcal{F}$ of $k$-graphs and every $\mu>0$ there exist constants $c > 0, C > 0$, and $n_0\in \mathbb{N}$ such that the following holds.
Suppose $H$ is a $k$-graph on $n\geq n_0$ vertices. If for every $\ell = 1,\ldots,C$ and every $F \in \mathcal{F}$ on $\ell$ vertices, $H$ contains at most $cn^{\ell}$ induced copies of $F$, then $H$ is $\mu$-near to $\text{Forb}_{ind}(F)$.
\end{theo}

\subsection{Tool 2: Flag Algebras}
\label{sec:flaga}
In this section we give an insight on how we apply Razborov's flag algebra machinery \cite{flagsRaz} for calculating the codegree squared density. 
The main power of the machinery comes from the possibility of formulating a problem as a semidefinite program and using a computer to solve it.

The method can be applied in various settings such as graphs~\cite{FAgraphs,FAgraphs2}, hypergraphs~\cite{BaberTalbot,codF32falgas}, oriented graphs~\cite{FATrounaments,FAdigraphs}, edge-coloured graphs~\cite{FAcolor,FAColor2}, permutations~\cite{PAPerm1,PAPerm2}, discrete geometry~\cite{FAGeometry,FAGeom2}, or phylogenetic trees~\cite{FAPhylogenetic}.
For a detailed explanation of the flag algebra method in the setting of $3$-uniform hypergraphs see \cite{RavryTuran}. Further, we recommend looking at the survey~\cite{MR3186665} and the expository note~\cite{MR3135939}, both by Razborov. 
Here, we will focus on the problem formulation rather than a formal explanation of the general method.

Let $F$ be a fixed $3$-graph. 
Let $\mathcal{F}$ denote the set of all $F$-free $3$-graphs up to isomorphism.
Denote by $\mathcal{F}_\ell$ all $3$-graphs in $\mathcal{F}$ on $\ell$ vertices.
For two $3$-graphs $F_1$ and $F_2$, denote by $P(F_1,F_2)$ the probability that $|V(F_1)|$ vertices chosen uniformly at random from $V(F_2)$ induce a copy of $F_1$.  
A sequence of $3$-graphs $(G_n)_{n \geq 1}$ of increasing orders is \emph{convergent}, if $\lim_{n \to \infty}P(H,G_n)$ exists for every $H \in \mathcal{F}$. Notice that if this limit exists, it is in $[0,1]$. 

For readers familiar with flag algebras and its usual notation, for a convergent sequence $(G_n)_{n \geq 1}$ of $n$-vertex $3$-graphs $G_n$, we get
\begin{align}\label{eq:fa}
\lim_{n \to \infty} \frac{\textup{co}_2(G_n)}{\binom{n}{2}(n-2)^2}
=
\left\llbracket
\left(
\vc{
\begin{tikzpicture}
\draw
(0,0) coordinate(1) node[vtxl,label=below:1](a){}
(1,0) coordinate(2) node[vtxl,label=below:2](b){}
(0.5,1) coordinate(3) node[vtx](c){}
;
\draw[hyperedge] (1) to[out=40,in=130] (2) to[out=120,in=280] (3) to[out=260,in=50] (1);
\draw
(1) node[vtxl]{}
(2) node[vtxl]{}
;
\end{tikzpicture}
}
\right)^2
\right\rrbracket_{1,2}
=
\frac{1}{6}
\vc{
\begin{tikzpicture}
\draw
(0,0) coordinate(1) node[vtx](a){}
(1,0) coordinate(2) node[vtx](b){}
(1,1) coordinate(3) node[vtx](c){}
(0,1) coordinate(4) node[vtx](d){}
;
\draw[hyperedge] (1) to[out=30,in=90] (2) to[out=100,in=260] (3) to[out=260,in=30] (1);
\draw[hyperedge] (1) to[out=90,in=150] (2) to[out=150,in=280] (4) to[out=280,in=90] (1);
\end{tikzpicture}
}
+
\frac{1}{2}
\Kfourthreeminus
+
\Kfourthree
,
\end{align}
where $\llbracket\cdot\rrbracket$ denotes the averaging operator and the terms on the right are interpreted as
\[
\lim_{n\to \infty} \frac{1}{6}P(K_4^{3=},G_n) + \frac{1}{2}P(K_4^{3-},G_n) + P(K_4^3,G_n), 
\]
where $K_4^{3=}$ is the $3$-graph with $4$ vertices and $2$ edges and $K_4^{3-}$ the $3$-graph with $4$ vertices and $3$ edges, also known as $F_4$.
It is a routine application of flag algebras to find an upper bound on the right-hand side of \eqref{eq:fa}.

For readers less familiar with flag algebras, the following paragraphs give a slightly less formal explanation of the problem formulation.
Let $G$ be a 3-graph. Let $\theta$ be an injective function $\{1,2\} \to V(G)$. 
In other words, $\theta$ labels two distinct vertices in $G$. 
We call the pair $(G,\theta)$ a \emph{labelled 3-graph} although only two vertices in $G$ are labelled by $\theta$.

Let $(H,\theta')$ and $(G,\theta)$ be two labelled $3$-graphs. 
Let $X$ be a subset of $V(G)\setminus \Img \theta$ of size $|V(H)|-2$ chosen uniformly at random.
By $P((H,\theta'),(G,\theta))$ we denote the probability that the labelled subgraph of $G$ induced by $X$ and the two labelled vertices, i.e., $(G[X\cup \Img \theta],\theta)$, is isomorphic to $(H,\theta')$, where the isomorphism maps $\theta(i)$ to $\theta'(i)$ for $i \in \{1,2\}$.

Let $E$ be a labelled 3-graph consisting of three vertices, two of them labelled, and one edge containing all three vertices.
Notice that $P(E,(G,\theta)) (n-2)$ is the codegree of $\theta(1)$ and $\theta(2)$ in a $3$-graph $G$.
The square of the codegree of $\theta(1)$ and $\theta(2)$ is $\left(P(E,(G,\theta)) (n-2)\right)^2$.
One of the tricks in flag algebras is that calculating $P(E,(G,\theta))^2$ in $G$ of order $n$ can be done within error $O(1/n)$ by selecting two distinct vertices in addition to $\theta(1)$ and $\theta(2)$ and examining subgraphs on four vertices instead.
In our case, it looks like the following, where $P(H, (G,\theta))$ is depicted simply as $H$.
\begin{align}\label{eq:faexpanded}
\left(
\vc{
\begin{tikzpicture}
\draw
(0,0) coordinate(1) node[vtxl,label=below:1](a){}
(1,0) coordinate(2) node[vtxl,label=below:2](b){}
(0.5,1) coordinate(3) node[vtx](c){}
;
\draw[hyperedge] (1) to[out=40,in=130] (2) to[out=120,in=280] (3) to[out=260,in=50] (1);
\draw
(1) node[vtxl]{}
(2) node[vtxl]{}
;
\end{tikzpicture}
}
\right)^2
=
\vc{
\begin{tikzpicture}
\draw
(0,0) coordinate(1) node[vtxl,label=below:1](a){}
(1,0) coordinate(2) node[vtxl,label=below:2](b){}
(1,1) coordinate(3) node[vtx](c){}
(0,1) coordinate(4) node[vtx](d){}
;
\draw[hyperedge] (1) to[out=30,in=90] (2) to[out=100,in=260] (3) to[out=260,in=30] (1);
\draw[hyperedge] (1) to[out=90,in=150] (2) to[out=150,in=280] (4) to[out=280,in=90] (1);
\draw
(1) node[vtxl]{}
(2) node[vtxl]{}
;
\end{tikzpicture}
}
+
\vc{
\begin{tikzpicture}
\draw
(0,0) coordinate(1) node[vtxl,label=below:1](a){}
(1,0) coordinate(2) node[vtxl,label=below:2](b){}
(1,1) coordinate(3) node[vtx](c){}
(0,1) coordinate(4) node[vtx](d){}
;
\draw[hyperedge] (1) to[out=30,in=90] (2) to[out=100,in=260] (3) to[out=260,in=30] (1);
\draw[hyperedge] (1) to[out=90,in=150] (2) to[out=150,in=280] (4) to[out=280,in=90] (1);
\draw[hyperedge] (1) to[out=60,in=200] (3) to[out=200,in=330] (4) to[out=330,in=60,looseness=1.5] (1);
\draw
(1) node[vtxl]{}
(2) node[vtxl]{}
;
\end{tikzpicture}
}
+
\vc{
\begin{tikzpicture}
\draw
(0,0) coordinate(1) node[vtxl,label=below:1](a){}
(1,0) coordinate(2) node[vtxl,label=below:2](b){}
(1,1) coordinate(3) node[vtx](c){}
(0,1) coordinate(4) node[vtx](d){}
;
\draw[hyperedge] (1) to[out=30,in=90] (2) to[out=100,in=260] (3) to[out=260,in=30] (1);
\draw[hyperedge] (1) to[out=90,in=150] (2) to[out=150,in=280] (4) to[out=280,in=90] (1);
\draw[hyperedge] (2) to[out=120,in=230,,looseness=1.8] (3) to[out=230,in=300,looseness=1.8] (4) to[out=300,in=120,looseness=1.8] (2);
\draw
(1) node[vtxl]{}
(2) node[vtxl]{}
;
\end{tikzpicture}
}
+
\vc{
\begin{tikzpicture}
\draw
(0,0) coordinate(1) node[vtxl,label=below:1](a){}
(1,0) coordinate(2) node[vtxl,label=below:2](b){}
(1,1) coordinate(3) node[vtx](c){}
(0,1) coordinate(4) node[vtx](d){}
;
\draw[hyperedge] (1) to[out=30,in=90] (2) to[out=100,in=260] (3) to[out=260,in=30] (1);
\draw[hyperedge] (1) to[out=90,in=150] (2) to[out=150,in=280] (4) to[out=280,in=90] (1);
\draw[hyperedge] (1) to[out=60,in=200] (3) to[out=200,in=330] (4) to[out=330,in=60,looseness=1.5] (1);
\draw[hyperedge] (2) to[out=120,in=230,,looseness=1.8] (3) to[out=230,in=300,looseness=1.8] (4) to[out=300,in=120,looseness=1.8] (2);
\draw
(1) node[vtxl]{}
(2) node[vtxl]{}
;
\end{tikzpicture}
}
+o(1)
\end{align}
The next step is to sum over all possible choices for $\theta$, there are $n(n-1)$ of them, and divide by $2$ since the codegree squared sum is over unordered pairs of vertices, unlike $\theta$. When summing over all possible $\theta$, one could look at all subsets of vertices of size $4$ of $G$ and see what the probability is that randomly labelling two vertices among these four by $\theta$ gives one of the labelled $3$-graphs from the right hand side of \eqref{eq:faexpanded}. 
This gives the coefficients on the right-hand side of \eqref{eq:fa}.

We use flag algebras to prove Lemmas~\ref{K43flag}, \ref{flag F33}, and \ref{K53flag}. 
The calculations are computer assisted. 
We use CSDP~\cite{csdp} for finding numerical solutions of semidefinite programs and SageMath~\cite{sagemath} for rounding the numerical solutions to exact ones.
The files needed to perform the corresponding calculations are available at \oururl.

\section{General results: Proofs of Propositions~\ref{coslim}, \ref{cossupersat} and \ref{blowup}}
\label{General results}

\subsection{The limit exists}
\begin{proof}[Proof of Proposition~\ref{coslim}]
Let $n\geq k$ be a positive integer and let $G$ be an $\mathcal{F}$-free $k$-graph on vertex set $[n]$ satisfying $\textup{co}_2(G)=\textup{exco}_2(n,\mathcal{F})$. Take $S$ to be a randomly chosen $(n-1)$-subset of $V(G)$. Now, we calculate the expectation of $\textup{co}_2(G[S])$,
\begin{align*}
\mathbb{E}[\textup{co}_2(G[S])] &= \sum_{T\in \binom{[n]}{k-1}}\mathbb{E}[ \mathbf{1}_{\{ T \subset S\}} d_{G[S]}^2(T)] =  \sum_{T\in \binom{[n]}{k-1}}\mathbb{P}(T\subset S) \mathbb{E}[ d_{G[S]}^2(T) \vert T\subset S ] \\ &= \sum_{T\in \binom{[n]}{k-1}}\frac{\binom{n-1}{k-1}}{\binom{n}{k-1}} \mathbb{E}[ d_{G[S]}^2(T) \vert T\subset S ] \geq \sum_{T\in \binom{[n]}{k-1}}\frac{\binom{n-1}{k-1}}{\binom{n}{k-1}} \mathbb{E}[ d_{G[S]}(T) \vert T\subset S ]^2 \\
&=  \sum_{T\in \binom{[n]}{k-1}} \frac{\binom{n-1}{k-1}}{\binom{n}{k-1}}\left(d_G(T) \frac{n-k}{n-k+1} \right)^2 = \frac{\binom{n-1}{k-1}}{\binom{n}{k-1}}\left(\frac{n-k}{n-k+1} \right)^2 \textup{co}_2(G).
\end{align*}
We used that $d_{G[S]}(T)$ conditioned on $T\subset S$ has hypergeometric distribution. By averaging, we conclude that there exists an $(n-1)$-vertex subset $S'\subset V(G)$ with $\textup{co}_2(G[S'])\geq \mathbb{E}[\textup{co}_2(G[S])]$. Thus, we conclude that $G[S']$ is an $(n-1)$-vertex $k$-graph satisfying 
\begin{align*}
\textup{co}_2(G[S'])\geq \frac{\binom{n-1}{k-1}}{\binom{n}{k-1}}\left(\frac{n-k}{n-k+1} \right)^2 \textup{co}_2(G).
\end{align*}
Therefore, since $G[S']$ is $\mathcal{F}$-free,
\[
\frac{\textup{exco}_2(n-1,\mathcal{F})}{\binom{n-1}{k-1}(n-k)^2}\geq \frac{\textup{co}_2(G[S'])}{\binom{n-1}{k-1}(n-k)^2}\geq \frac{\textup{co}_2(G)}{\binom{n}{k-1}(n-k+1)^2}= \frac{\textup{exco}_2(n,\mathcal{F})}{\binom{n}{k-1}(n-k+1)^2}.
\qedhere
\]
\end{proof}

\subsection{Supersaturation}
In this section we prove Proposition~\ref{cossupersat}. We will make use of the following tail bound on the hypergeometric distribution. 
\begin{lemma}[e.g.~\cite{Jansonhypergeometric}{ p.29}]
\label{hypergeometrictails}
Let $\beta,\lambda>0$ with $\beta+\lambda<1$. Suppose that $X\subseteq [n]$ and $|X|\geq (\beta+\lambda)n$. Then
\begin{align*}
\left\vert \left\{ S\in \binom{[n]}{m}: |S \cap X|\leq \beta m \right\} \right\vert \leq \binom{n}{m}e^{-\frac{\lambda^2m}{3(\beta+\lambda)}}\leq \binom{n}{m}e^{-\lambda^2m/3}.
\end{align*}
\end{lemma}

Mubayi and Zhao \cite{Mubayicode} used Lemma~\ref{hypergeometrictails} to prove a supersaturation result for the minimum codegree threshold. We adapt their proof to our setting.
\begin{lemma}
\label{supersatlemma}
Let $\alpha>0$, $\varepsilon>0$ and $k\geq 3$. Then there exists $m_0$ such that the following holds. If $n\geq m \geq m_0$ and $G$ is a $k$-graph on $[n]$ with $\textup{co}_2(G)\geq (\alpha+\varepsilon) \binom{n}{k-1}(n-k+1)^2$, then the number of $m$-sets $S$ satisfying $\textup{co}_2(G[S]) > \alpha \binom{m}{k-1}(m-k+1)^2$ is at least $\frac{\varepsilon}{4} \binom{n}{m}$.   \end{lemma}

\begin{proof}
Given a $(k-1)$-element set $T\subset [n]$, we call an $m$-set $S$ with $T \subset S \subset [n]$ \textit{bad for} $T$ if $|d(T)\cap S|\leq \left( \frac{d(T)}{n-k+1}- \frac{\varepsilon}{6}\right) (m-k+1)$. An $m$-set is \textit{bad} if it is bad for some $T$. Otherwise, it is \textit{good}. We will show that there are only few bad sets. Denote by $\Phi$ the number of bad $m$-sets, and let $\Phi_T$ be the number of $m$-sets that are bad for $T$. Then, by applying Lemma~\ref{hypergeometrictails} with $\beta= \frac{d(T)}{n-k+1}- \frac{\varepsilon}{6}$ and $\lambda=\varepsilon/7$, we get
\begin{align*}
\Phi &\leq  \sum_{T \in \binom{[n]}{k-1}} \Phi_T = \sum_{T \in \binom{[n]}{k-1}} \left\vert \left\{ S'\in \binom{[n]\setminus T}{m-k+1}: |d(T) \cap S'|\leq \left( \frac{d(T)}{n-k+1}- \frac{\varepsilon}{6}\right) (m-k+1) \right\} \right\vert \\
&\leq \sum_{T \in \binom{[n]}{k-1}} \binom{n-k+1}{m-k+1} \exp \left(-\frac{\varepsilon^2(m-k+1)}{147}  \right) \leq 
 \binom{n}{k-1}\binom{n-k+1}{m-k+1} \exp \left(-\frac{\varepsilon^2(m-k+1)}{147}  \right) \\
 &= \binom{n}{m}\binom{m}{k-1} \exp \left(-\frac{\varepsilon^2(m-k+1)}{147}  \right) \leq \frac{\varepsilon}{4}\binom{n}{m},
\end{align*}
where the last inequality holds for $m$ large enough.
So the number of bad $m$-sets is at most $\frac{\varepsilon}{4}\binom{n}{m}$. 
Now let $\ell \binom{n}{m}$ be the number of $m$-sets $S$ satisfying 
\begin{align}
\label{good*mset}
    \sum_{T \in \binom{S}{k-1}} d_G^2(T)\geq \left( \alpha +\frac{\varepsilon}{2}\right) \binom{m}{k-1}(n-k+1)^2.  
\end{align}
On one side
\begin{align*}
\sum_{ |S|=m}  \sum_{T \in \binom{S}{k-1}} d_G^2(T)= \binom{n-k+1}{m-k+1} \textup{co}_2(G) = \binom{n-k+1}{m-k+1}\binom{n}{k-1}(n-k+1)^2 (\alpha+\varepsilon).
\end{align*}
On the other side,
\begin{align*}
\sum_{|S|=m}  \sum_{T \in \binom{S}{k-1}} d_G^2(T) &\leq \left(\alpha+\frac{\varepsilon}{2}\right) \binom{m}{k-1}(n-k+1)^2 \binom{n}{m} + \ell \binom{m}{k-1}(n-k+1)^2 \binom{n}{m}\\
&= \left(\alpha+\frac{\varepsilon}{2}+\ell\right) \binom{m}{k-1}(n-k+1)^2 \binom{n}{m}.
\end{align*}
By this double counting argument, we conclude $\ell\geq \varepsilon/2$. Since the number of bad $m$-sets is at most $\frac{\varepsilon}{4}\binom{n}{m}$, there are at least $\frac{\varepsilon}{4}\binom{n}{m}$ good $m$-sets satisfying \eqref{good*mset}. All of these $m$-sets satisfy

\begin{align*}
    \textup{co}_2(G[S]) &=\sum_{T \in \binom{S}{k-1}} d_{G[S]}^2(T) \geq \sum_{T \in \binom{S}{k-1}} \left( \left( \frac{d_G(T)}{n-k+1}- \frac{\varepsilon}{6}\right) (m-k+1)\right)^2 \\
    &= \frac{(m-k+1)^2}{(n-k+1)^2} \sum_{T \in \binom{S}{k-1}} \left( d_G(T)- \frac{\varepsilon}{6}(n-k+1)\right)^2  \\ &\geq 
    \frac{(m-k+1)^2}{(n-k+1)^2} \sum_{T \in \binom{S}{k-1}} \left(d_G^2(T)- \frac{\varepsilon}{3}(n-k+1)^2\right)\\
    &\geq \frac{(m-k+1)^2}{(n-k+1)^2}  \left( \left( \alpha +\frac{\varepsilon}{2}\right) \binom{m}{k-1}(n-k+1)^2 - \binom{m}{k-1}\frac{\varepsilon}{3}(n-k+1)^2 \right) \\ &> \alpha \binom{m}{k-1}(m-k+1)^2,
\end{align*}
proving the statement of this lemma.
\end{proof}
\begin{proof}[Proof of Proposition~\ref{cossupersat}]
This proof follows Erd\H{o}s and Simonovits's proof \cite{Supersaterdos} of the supersa\-tu\-ration result for the Tur\'an density.  \par
Let $F$ be a $k$-graph on $f$ vertices, $\varepsilon>0$ and $G$ be an $n$-vertex $k$-graph satisfying $\textup{co}_2(G)>(\sigma(F)+\varepsilon) \binom{n}{k-1}n^2$ for $n$ large enough. By Lemma~\ref{supersatlemma}, there exists an $m_0$ such that for $m\geq m_0$ the number of $m$-sets $S$ satisfying $\textup{co}_2(G[S])>(\sigma(F)+\varepsilon/2) \binom{m}{k-1}(m-k+1)^2$ is at least $\frac{\varepsilon}{8} \binom{n}{m}$. There exists some fixed $m_1\geq m_0$ such that $\textup{exco}_2(m_1,F)\leq (\sigma(F)+\varepsilon/2) \binom{m_1}{k-1}(m_1-k+1)^2$. Thus, there are at least $\frac{\varepsilon}{8} \binom{n}{m_1}$ $m_1$-sets $S$ such that $G[S]$ contains $F$. Each copy of $F$ may be counted at most $\binom{n-f}{m_1-f}$ times. Therefore, the number of copies for $F$ is at least
\begin{align*}
    \frac{\frac{\varepsilon}{8} \binom{n}{m_1}}{\binom{n-f}{m_1-f}}=\delta \binom{n}{f},
\end{align*}
for $\delta=\frac{\varepsilon}{8\binom{m_1}{f}}$.
\end{proof}

\subsection{Proof of Corollary~\ref{blowup} and Proposition~\ref{jumps}}
Now we use a standard argument to show that blowing-up a $k$-graph does not change the codegree squared density. We will follow the proof of the analogous Tur\'an result given in \cite{Keevashsurvey}. 
\begin{proof}[Proof of Corollary~\ref{blowup}]
Since $H\subset H(t)$, $\textup{exco}_2(n,H(t))\leq \textup{exco}_2(n,H)$ holds trivially. Thus, $\sigma(H(t))\leq \sigma(H)$.

For the other direction, let $\varepsilon>0$ and $G$ be an $n$-vertex $k$-uniform hypergraph satisfying $\textup{co}_2(G)/(\binom{n}{k-1}(n-k+1)^2)> \sigma(H)+\varepsilon$. Then, by Proposition~\ref{cossupersat}, $G$ contains at least $\delta \binom{n}{v(H)}$ copies of $H$ for $\delta=\delta(\varepsilon,k)>0$. We create an auxiliary $v(H)$-graph $F$ on the vertex set $V(G)$. A $v(H)$-set $A\subset V(G)$ is an edge in $F$ iff $G[A]$ contains a copy of $H$. The auxiliary hypergraph $F$ has density at least $\delta/v(H)!$. Thus, as it is well-known~\cite{Erdos3partite}, for any $t'>0$ as long as $n$ is large enough, $F$ contains a copy of $K_{v(H)}^{v(H)}(t')$, the complete $v(H)$-partite $v(H)$-graph with $t'$ vertices in each part. We choose $t'$ large enough such that the following is true. We colour each edge of $K_{v(H)}^{v(H)}(t')$ by one of $v(H)!$ colours, depending on which of the $v(H)!$ orders the vertices of $H$ are mapped to in the corresponding copy of $H$ in $G$. By a classical result in Ramsey theory (for a density version see \cite{Erdos3partite}), there is a monochromatic copy of $K_{v(H)}^{v(H)}(t)$, which contains a copy of $H(t)$ in $G$. We conclude $\sigma(H(t))\leq \sigma(H)+\varepsilon$ for all $\varepsilon>0$.      
\end{proof}

\begin{proof}[Proof of Proposition~\ref{jumps}]
Let $H$ be a $k$-graph. For any $k$-graph $G$, we have by the Cauchy-Schwarz inequality
\begin{align*}
    \textup{co}_2(G)= \sum_{T\in \binom{[n]}{k-1}}d_G(T)^2\geq \frac{\left(\sum_{T\in \binom{[n]}{k-1}}d_G(T)\right)^2}{\binom{n}{k-1}}=\frac{\left(k|E(G)|\right)^2}{\binom{n}{k-1}}.
\end{align*}
Applying this for an $H$-free hypergraph $G$, and scaling, we obtain $\sigma(H)\geq \pi(H)^2$. For \\ $\sigma(H)\leq\pi(H)$ we use
\begin{align*}
\textup{co}_2(G)= \sum_{T\in \binom{[n]}{k-1}}d_G(T)^2= \sum_{e\in E(G)}w_G(e) \leq kn |E(G)|,
\end{align*}
where $w_G(e):=\sum_{T\in \binom{e}{k-1}}d_G(T)$. After scaling this implies $\sigma(H)\leq \pi(H)$, completing the proof of part $(i)$. 

Erd\H{o}s~\cite{Erdos3partite} proved that the Tur\'an density of a $k$-partite $k$-graph is $0$. In this case, the codegree squared density is also $0$ by part (i). If $H$ is not $k$-partite, then the complete $k$-partite hypergraph is $H$-free providing a construction for lower bounds. Hence, as it was observed by Erd\H{o}s~\cite{Erdos3partite}, the Tur\'an density of $H$ is at least $k!/k^k$. Similarly, we get $\sigma(H)\geq (k-1)/k^k$.
\end{proof}


\section{Cliques}
\label{cliques}
In this section we will prove Theorems~\ref{K43stability} and \ref{K53stability}.

\subsection{Proof of Theorem~\ref{K43stability}}
Flag algebras give us the following results for $K_4^3$.
\begin{lemma}
\label{K43flag}
For all $\varepsilon>0$ there exists $\delta>0$ and $n_0$ such that for all $n\geq n_0$: if $G$ is a $K_4^3$-free $3$-uniform graph on $n$ vertices with $\textup{co}_2(G)\geq (1-\delta)\frac{1}{3} n^4/2$, then the densities of all $3$-graphs on $4,5$ and $6$ vertices in $G$ that are not contained in $C_n$ are at most $\varepsilon$. Additionally, 
\begin{align*}
    \sigma(K_4^3)=\frac{1}{3}.
\end{align*}
\end{lemma}
The flag algebra calculation proving Lemma~\ref{K43flag} is computer assisted.
The calculation is available at \oururl. For proving Theorem~\ref{K43stability} we will make use of the following stability result due to Pikhurko~\cite{PikhurkoK43}. 
\begin{theo}[Pikhurko~\cite{PikhurkoK43}]
\label{K43stabilityPik}
For every $\varepsilon>0$ there exists $\delta>0$ and $n_0$ such that for every $n > n_0$, if $G$ is a $K_4^3$-free $3$-uniform hypergraph on $n$ vertices not spanning exactly one edge on four vertices and with 
\begin{align*}
e(G)\geq \left(\frac{5}{9}-\delta \right) \binom{n}{3}, 
\end{align*}
then $G$ is $\varepsilon$-near to $C_n$.
\end{theo}
\begin{proof}[Proof of Theorem~\ref{K43stability}]
Let $\varepsilon>0$ be fixed. We choose $n_0$ sufficiently large for the following proof to work. We will choose constants \begin{align*}
    1 \gg \varepsilon \gg \delta_3 \gg \delta_2 \gg \delta_1 \gg \delta \gg 0 
\end{align*}
in order from left to right where each constant is a sufficiently small positive number depending only on the previous ones. Let $G$ be a $K_4^3$-free $3$-uniform hypergraph on $n\geq n_0$ vertices with \begin{align*}
\textup{co}_2(G)\geq \left(\frac{1}{3}-\delta \right) \frac{n^4}{2}.
\end{align*} 
By applying Lemma~\ref{K43flag}, we get that the density of the 4-vertex 3-graph with exactly one edge in $G$ is at most $\delta_1$. Now, we apply the induced hypergraph removal lemma, Theorem~\ref{removallemma}, to obtain $G'$ where $G'$ is $\delta_2$-near to $G$, and $G'$ is $K_4^3$-free and does not induce exactly one edge on four vertices. We have
\begin{align*}
\textup{co}_2(G')\geq \textup{co}_2(G)- 6\delta_2 n^4\geq  \left(\frac{1}{3}-\delta \right) \frac{n^4}{2} -6\delta_2 n^4 \geq (1-37\delta_2)\frac{1}{6}n^4,
\end{align*}
where the first inequality holds because when one edge is removed from a $3$-uniform hypergraph, then the codegree squared sum can go down by at most $6n$. By a result of Falgas-Ravry and Vaughan~\cite[Theorem 4]{MR2988862}, $P(K_4^{3-},G') \leq16/27+o(1)$. Let $x\in [0,1]$ such that $P(K_4^{3-},G')=16/27(1-x)+o(1)$. By \eqref{eq:fa} and the fact that $G'$ is $K_4^3$-free, we have 
\begin{align*}
    \frac{1}{3}(1-37\delta_2)\leq \frac{\textup{co}_2(G')}{\binom{n}{2}(n-2)^2}=\frac{1}{6}P(K_4^{3=},G') + \frac{1}{2}P(K_4^{3-},G')\leq \frac{1}{6}P(K_4^{3=},G') +\frac{8}{27}(1-x)+\delta_2.
\end{align*}
Thus, 
\begin{align}
\label{eq:K4G'}
P(K_4^{3=},G')\geq \frac{2+16x}{9}-80\delta_2.
\end{align}
Since $G'$ does not contain a $4$-set spanning exactly $1$ or $4$ edges, a result of Razborov~\cite{RazbarovK43} says
\begin{align}
\label{pik}
 \frac{|E(G')|}{\binom{n}{3}}\leq \frac{5}{9}+o(1).
\end{align}
The edge density can also expressed as
\begin{align}
\label{edgedensityg'}
    \frac{|E(G')|}{\binom{n}{3}}=  \frac{1}{2}P(K_4^{3=},G')+\frac{3}{4}P(K_4^{3-},G')+o(1).
\end{align}
By combining \eqref{eq:K4G'} and \eqref{edgedensityg'} we get
\begin{align*}
\label{}
 \frac{|E(G')|}{\binom{n}{3}}\geq  \frac{1}{2}P(K_4^{3=},G')+\frac{3}{4}P(K_4^{3-},G')-\delta_2\geq \frac{5+4x}{9}-41\delta_2.
\end{align*}
This implies that $x\leq 100\delta_2$. Thus, by Pikhurko's stability theorem (Theorem~\ref{K43stabilityPik}), $G'$ is $\delta_3$-near to $C_n$. Since $G'$ is $\delta_2$-near to $G$, we conclude that $G$ is $\varepsilon$-near to $C_n$. 
\end{proof}

\subsection{Proof of Theorem~\ref{K53stability}}
Flag algebras give us the following for $K_5^3$.  
\begin{lemma}
\label{K53flag}
For all $\varepsilon>0$ there exists $\delta>0$ and $n_0$ such that for all $n\geq n_0$: if $G$ is a $K_5^3$-free $3$-uniform graph on $n$ vertices with $\textup{co}_2(G)\geq (1-\delta)\frac{5}{8} n^4/2$, then the densities of all $3$-graphs on $4,5$ and $6$ vertices in $G$ that are not contained in $B_n$ are at most $\varepsilon$. In particular, 
\begin{align*}
    \sigma(K_5^3)=\frac{5}{8}.
\end{align*}
\end{lemma}
Again, the flag algebra calculation proving Lemma~\ref{K53flag} is computer assisted and available at \oururl. We use this result to prove Theorem~\ref{K53stability}.  

\begin{proof}[Proof of Theorem~\ref{K53stability}]
Let $\varepsilon>0$. During the proof we will use the following constants: 
\begin{align*}
    1 \gg \varepsilon \gg \delta_2 \gg \delta_1 \gg \delta \gg 0. 
\end{align*}
The constants are chosen in this order and each constant is a sufficiently small positive number depending only on the previous ones. Apply Lemma~\ref{K53flag} and get $\delta=\delta(\delta_1)>0$ such that for all $n$ large enough: If $G$ is a $K_5^3$-free $3$-uniform graph on $n$ vertices with $\textup{co}_2(G)\geq (1-\delta)\frac{5}{8} n^4/2$, then the densities of all $3$-graphs on $4,5$ and $6$ vertices in $G$ that are not contained in $B_n$ are at most $\delta_1$.

Now, apply the induced hypergraph removal lemma Theorem~\ref{removallemma} to obtain $G'$ where $G'$ is $\delta_2$-near to $G$, and $G'$ contains only those induced subgraphs on $4,5$ or $6$ vertices which appear as induced subgraphs in $B_n$. Note that 
\begin{align*}
\textup{co}_2(G')\geq \textup{co}_2(G')- 6\delta_2 n^4\geq  (1-\delta)\frac{5}{8}\frac{n^4}{2} -6\delta_2 n^4 \geq (1-20\delta_2)\frac{5}{8}\frac{n^4}{2},
\end{align*}
because when one edge is removed the codegree squared sum can go down by at most $6n$. Next we show that $G'$ has to have the same structure as $B_n$. We say that a $3$-graph $H$ is \emph{$2$-colourable}, if there is a partition of the vertex set $V(H)=V_1 \cup V_2$ such that $V_1$ and $V_2$ are independent sets in $H$.
\begin{claim}
\label{H2color}
$G'$ is $2$-colourable.
\end{claim}
\begin{proof}
Take an arbitrary non-edge $abc$ in $G'$. For $0\leq i \leq 4$, define $A_i$ to be the set of vertices $v\in V(G)\setminus \{a,b,c\}$ such that $G'$ induces $i$ edges on $\{a,b,c,v\}$. Then, $A_1=A_2=\emptyset$ because on $4$ vertices there are either $0,3$ or $4$ edges in $B_n$, hence in $G'$ as well. Further $A_4=\emptyset$, because $abc$ is a non-edge. Clearly, $A_0$ is an independent set, because if there is an edge $v_1v_2v_3$ in $G'[A_0]$, then the induced graph of $G'$ on $\{a,b,c,v_1,v_2,v_3\}$ spans a forbidden subgraph, i.e., a hypergraph which is not an induced subhypergraph of $B_n$. Similarly, $A_3$ is an independent set, otherwise $G'$ were to contain a copy of $F_{3,3}$, which is not an induced subhypergraph of $B_n$. Let $A'=A_0 \cup \{a,b,c\}$.~Then $V(G')=A_3 \cup A'$ and $A'$ also forms an independent set. To observe the second statement, let $v_1,v_2,v_3$ be three vertices in $A_0$. The number of edges induced on $\{v_1,v_2,v_3,a,b,c\}$ is at most nine, because every edge needs to be incident to exactly two vertices of $\{a,b,c\}$ by the definition of $A_0$. However, $6$-vertex induced subgraphs of $B_n$ have either $0,10,16$, or $18$ edges. We conclude that $\{v_1,v_2,v_3,a,b,c\}$ induces no edge in $G'$. Thus, $A'$ is also an independent set in $G'$ and therefore $G'$ is $2$-colourable.  
\end{proof}
\begin{claim}
\label{edgesH'1}
We have $|E(G')|\geq (1-2\sqrt{\delta_2}) \frac{n^3}{8}$.
\end{claim}
\begin{proof}
By Claim~\ref{H2color}, $G'$ is $2$-colourable and we can partition the vertex set $V(G')=A\cup B$ such that $A$ and $B$ are independent sets. Let $a\in[0,1]$ such that $|A|=an$ and $|B|=(1-a)n$. We have 
\begin{align*}
    \left(1-20\delta_2\right)\frac{5}{8}\frac{n^4}{2} \leq \textup{co}_2(G') \leq \left( \frac{a^2}{2} (1-a)^2 + \frac{(1-a)^2}{2}a^2+a(1-a)\right) n^4 \leq \frac{5}{4}a(1-a)n^4. 
\end{align*}
Thus, $4a(1-a)\geq 1-20\delta_2$. We conclude        $1/2-3\sqrt{\delta_2}\leq a \leq 1/2+3\sqrt{\delta_2}$, otherwise 
\begin{align*}
    4a(1-a)< 4 \left( \frac{1}{2}-3\sqrt{\delta_2}\right) \left( \frac{1}{2}+3\sqrt{\delta_2}\right) = 1 - 36\delta_2,
\end{align*}
a contradiction. For every edge $e \in E(G')$, we have $w_{G'}(e)\leq \left(5/2 +3\sqrt{\delta_2}\right)n.$ Therefore,
\begin{align*}
    \left(1-20\delta_2\right)\frac{5}{8}\frac{n^4}{2} \leq \textup{co}_2(G')= \sum_{e\in E(G')}w_{G'}(e)\leq |E(G')|\left(\frac{5}{2} +3\sqrt{\delta_2}\right)n.
\end{align*}
Thus,
\begin{align*}|E(G')|\geq \frac{\left(1-20\delta_2\right)}{\left(1 +\frac{6}{5}\sqrt{\delta_2}\right)}\frac{n^3}{8}\geq (1-2\sqrt{\delta_2}) \frac{n^3}{8}. 
\end{align*}
\end{proof}
The $3$-graph $G$ is $\delta_2$-near to $G'$. By Claims~\ref{H2color} and \ref{edgesH'1}, $G'$ is $\varepsilon/2$-near to $B_n$. Therefore we can conclude that $G$ is $\delta_2 + \varepsilon/2 \leq \varepsilon$-near to $B_n$.   
\end{proof}

\subsection{Discussion on Cliques}
Keevash and Mubayi~\cite{Keevashsurvey} constructed the following family of $3$-graphs obtaining the best-known lower bound for the Tur\'an density of cliques. Denote by $\mathcal{D}_k$ the family of directed graphs on $k-1$ vertices that are unions of vertex-disjoint directed cycles. Cycles of length two are allowed, but loops are not. Let $D\in \mathcal{D}_k$ and $V=[n]=V_1\cup \ldots \cup V_{k-1}$ be a vertex partition with class sizes as balanced as possible, that is $||V_i|-|V_j||\leq 1$ for all $i\neq j$. Denote by $G(D)$ the $3$-graph on $V$ where a triple is a non-edge iff it is contained in some $V_i$ or if it has two vertices in $V_i$ and one vertex in $V_j$ where $(i,j)$ is an arc of $D$. The 3-graph $G(D)$ is $K_{k}^3$-free and has edge density $1-(2/t)^2+o(1)$. 
While all directed graphs $D\in\mathcal{D}_k$ give the same edge density for $G(D)$, up to isomorphism there is only one $D$ maximising the codegree squared sum $\textup{co}_2(G(D))$.
Let $D^*_k\in \mathcal{D}_k$ be the directed graph on $k-1$ vertices $v_1,\ldots,v_{k-1}$ such that if $k$ odd, then 
\begin{align*}
    (v_iv_{i+1}),(v_{i+1}v_i)\in E(D^*_k) \quad \text{for all odd $i$,}
\end{align*}
and if $k$ even, then
\begin{gather*}
    (v_iv_{i+1}),(v_{i+1}v_i)\in E(D^*_k) \quad \text{for all odd $i\leq k-5$} \\
    \text{and} \quad (v_{k-3}v_{k-2}),(v_{k-2}v_{k-1}),(v_{k-1}v_{k-3})\in E(D^*_k).
\end{gather*}
Note that $D^*_k$ is maximising the number of directed cycles. The 3-graph $G(D^*_4)$ is isomorphic to $C_n$ and $G(D^*_5)$ is isomorphic to $B_n$. See Figure~\ref{fig:Gl} for a drawing of $D^*_7,D^*_8$ and the complements $\overline{G(D^*_{7})}$ and $\overline{G(D^*_{8})}$ of $G(D^*_{7})$ and $G(D^*_{8})$, respectively. Next, we observe that among all directed graphs $D\in \mathcal{D}_k$, $D^*_k$ maximises the codegree squared sum of $G(D)$.

\begin{figure}[ht]
    \begin{center}
\tikzset{
vtx/.style={inner sep=1.1pt, outer sep=0pt, circle, fill,draw}, 
hyperedge/.style={fill=pink,opacity=0.5,draw=black}, 
vtxBig/.style={inner sep=15pt, outer sep=0pt, circle, fill=white,draw}, 
vtxBigC/.style={inner sep=12pt, outer sep=0pt, circle, fill=white,draw}, 
hyperedge/.style={fill=pink,opacity=0.5,draw=black}, 
}
\vc{
\begin{tikzpicture}[scale=1.4]
\begin{scope}[scale=1.3]
\draw (0,1) coordinate(x1) node[vtxBig]{};
\draw (0,0) coordinate(x2) node[vtxBig]{};
\draw (1,1) coordinate(x3) node[vtxBig]{};
\draw (1,0) coordinate(x4) node[vtxBig]{};
\draw (2,1) coordinate(x5) node[vtxBig]{};
\draw (2,0) coordinate(x6) node[vtxBig]{};
\end{scope}
\foreach \i/\j in {2/1,4/3,6/5}{
\draw(x\i)
+(270:0.3) coordinate(1) 
+(270+120:0.3) coordinate(2) 
+(270+240:0.3) coordinate(3) 
;
\draw[hyperedge] (1) to[bend left] (2) to[bend left] (3) to[bend left] (1);
\draw(x\j)
+(90:0.3) coordinate(1) 
+(90+120:0.3) coordinate(2) 
+(90+240:0.3) coordinate(3) 
;
\draw[hyperedge] (1) to[bend left] (2) to[bend left] (3) to[bend left] (1);
\draw (x\i)++(-0.2,-0.2)  node{$V_\i$};
\draw (x\j)++(-0.2,0.2)  node{$V_\j$};
\draw (x\i) +(-0.1,0.2) coordinate(1) 
(x\i) +(-0.3,0.2) coordinate(2) 
(x\j) +(-0.2,-0.2) coordinate(3) 
;
\draw[hyperedge] (1) to[bend right] (2) to[bend right=5] (3) to[bend right=5] (1);
\draw (x\j) +(0.1,-0.2) coordinate(1) 
(x\j) +(0.3,-0.2) coordinate(2) 
(x\i) +(+0.2,0.2) coordinate(3) 
;
\draw[hyperedge] (1) to[bend right] (2) to[bend right=5] (3) to[bend right=5] (1);
}

\begin{scope}[xshift = -5cm,yscale=1.2]
\draw (0,1) node[vtx,label=above:$v_1$](x1){};
\draw (0,0) node[vtx,label=below:$v_2$](x2){};
\draw (1,1) node[vtx,label=above:$v_3$](x3){};
\draw (1,0) node[vtx,label=below:$v_4$](x4){};
\draw (2,1) node[vtx,label=above:$v_5$](x5){};
\draw (2,0) node[vtx,label=below:$v_6$](x6){};
\foreach \i/\j in {1/2,3/4,5/6}{
\draw[-latex] (x\i) to[bend left] (x\j);
\draw[-latex] (x\j) to[bend left] (x\i);
}
\end{scope}
\draw(-4,-1) node{$D_7^*$};
\draw(1.5,-1) node{$\overline{G(D_7^*)}$};

\begin{scope}[yshift=-3.5cm]

\draw(-4,-1) node{$D_8^*$};
\draw(1.5,-1) node{$\overline{G(D_8^*)}$};

\begin{scope}[scale=1.3]
\draw (0,1) coordinate(x1) node[vtxBig]{};
\draw (0,0) coordinate(x2) node[vtxBig]{};
\draw (1,1) coordinate(x3) node[vtxBig]{};
\draw (1,0) coordinate(x4) node[vtxBig]{};
\end{scope}
\foreach \i/\j in {2/1,4/3}{
\draw(x\i)
+(270:0.3) coordinate(1) 
+(270+120:0.3) coordinate(2) 
+(270+240:0.3) coordinate(3) 
;
\draw[hyperedge] (1) to[bend left] (2) to[bend left] (3) to[bend left] (1);
\draw(x\j)
+(90:0.3) coordinate(1) 
+(90+120:0.3) coordinate(2) 
+(90+240:0.3) coordinate(3) 
;
\draw[hyperedge] (1) to[bend left] (2) to[bend left] (3) to[bend left] (1);
\draw (x\i)++(-0.2,-0.2)  node{$V_\i$};
\draw (x\j)++(-0.2,0.2)  node{$V_\j$};

\draw (x\i) +(-0.1,0.2) coordinate(1) 
(x\i) +(-0.3,0.2) coordinate(2) 
(x\j) +(-0.2,-0.2) coordinate(3) 
;
\draw[hyperedge] (1) to[bend right] (2) to[bend right=5] (3) to[bend right=5] (1);

\draw (x\j) +(0.1,-0.2) coordinate(1) 
(x\j) +(0.3,-0.2) coordinate(2) 
(x\i) +(+0.2,0.2) coordinate(3) 
;
\draw[hyperedge] (1) to[bend right] (2) to[bend right=5] (3) to[bend right=5] (1);
}

\begin{scope}[xshift=3cm,yshift=0.5*1.3cm]
\begin{scope}[rotate=90]
\draw (30:0.75) coordinate(x1) node[vtxBig]{};
\draw (150:0.75) coordinate(x2) node[vtxBig]{};
\draw (270:0.75) coordinate(x3) node[vtxBig]{};
\foreach \ashift in {30,150,270}{
\draw
(\ashift:0.4)++(\ashift+60:0.1) coordinate(1) 
++(\ashift+60:0.3) coordinate(2) 
(\ashift+120:0.4)++(\ashift+120+90:-0.2) coordinate(3)
;
\draw[hyperedge] (1) to[bend left] (2) to[bend left=5] (3) to[bend left=5] (1);
\draw
(\ashift:0.7)
+(\ashift+190:0.3) coordinate(1) 
+(\ashift+190+120:0.3) coordinate(2) 
+(\ashift+190+240:0.3) coordinate(3) 
;
\draw[hyperedge] (1) to[bend left] (2) to[bend left] (3) to[bend left] (1);
}
\draw
(30:1) node{$V_5$}
(150:1) node{$V_6$}
(270:1) node{$V_7$}
;
\end{scope}
\end{scope}

\begin{scope}[xshift = -5cm,yscale=1.2]
\draw (0,1) node[vtx,label=above:$v_1$](x1){};
\draw (0,0) node[vtx,label=below:$v_2$](x2){};
\draw (1,1) node[vtx,label=above:$v_3$](x3){};
\draw (1,0) node[vtx,label=below:$v_4$](x4){};
\draw (2,1) node[vtx,label=above:$v_5$](x5){};
\draw (2,0) node[vtx,label=below:$v_6$](x6){};
\draw (3,0.5) node[vtx,label=right:$v_7$](x7){};
\foreach \i/\j in {1/2,3/4}{
\draw[-latex] (x\i) to[bend left] (x\j);
\draw[-latex] (x\j) to[bend left] (x\i);
}
\draw[-latex] (x5) to[bend right] (x6);
\draw[-latex] (x6) to[bend right] (x7);
\draw[-latex] (x7) to[bend right] (x5);
\end{scope}

\end{scope}

\end{tikzpicture}
}

    \end{center}
    \caption{Representations of $D^*_{7},D^*_8$ and the complements $\overline{G(D^*_{7})}$ and $\overline{G(D^*_{8})}$ of $G(D^*_{7})$ and $G(D^*_{8})$, respectively.} 
    \label{fig:Gl}
\end{figure}
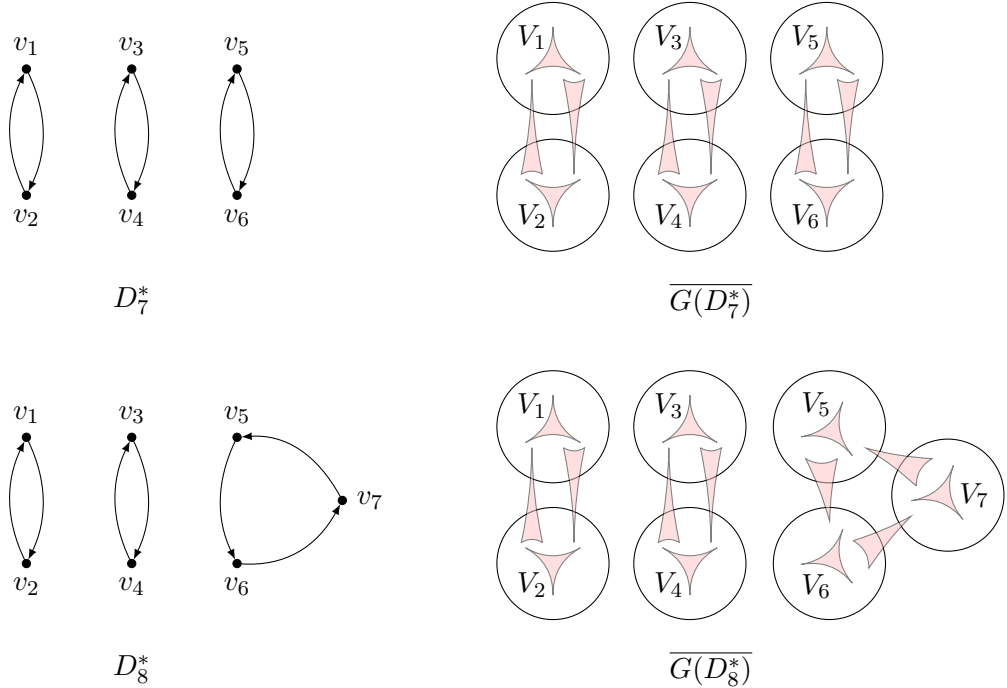
For a function $f: X \rightarrow \mathbb{R}$, and $S\subseteq X$, define 
\begin{align*}
\argmax_{x\in S}f(x):=\left\{x\in S : f(s)\leq f(x) \text{ for all } s\in S\right\}.    
\end{align*}
\begin{lemma}
\label{mubayiconstructions}
Let $k\geq 4$. For $n$ sufficiently large, $D^*_k$ is isomorphic to any directed graph in
\begin{align*}
    \argmax_{D \in \mathcal{D}_k}\textup{co}_2(G(D)).
\end{align*}
\end{lemma}
\begin{proof}
 Let $D\in \argmax_{D \in \mathcal{D}}co_2(G(D))$. Suppose for contradiction that $D$ contains a directed cycle $v_1,v_2,\ldots,v_\ell$ of length $\ell\geq 4$. Construct a directed graph $D'$ by replacing that $\ell$-cycle with an $(\ell-2)$-cycle $v_1,v_4,\ldots,v_{\ell-2}$ and a 2-cycle $v_{2},v_{3}$. Let $V_1,V_2,\ldots,V_\ell$ be the corresponding classes in $G$. The only pairs of vertices $x,y$ for which the codegree changes by more than $O(1)$ are described in the following. 
 \begin{itemize}
     \item For $x\in V_{1},y\in V_{2}$, $d(x,y)$ increased from $n-n/(k-1)+O(1)$ to $n+O(1)$.
     \item For $x\in V_{3},y\in V_{4}$, $d(x,y)$ increased from $n-n/(k-1)+O(1)$ to $n+O(1)$.
     \item For $x\in V_{2},y\in V_{3}$, $d(x,y)$ decreased from $n-n/(k-1)+O(1)$ to $n-2n/(k-1)+O(1)$.
     \item For $x\in V_{1},y\in V_{4}$, $d(x,y)$ decreased from $n-n/(k-1)+O(1)$ to $n-2n/(k-1)+O(1)$ if $\ell=4$ or 
     from $n+O(1)$ to $n-n/(k-1)+O(1)$ if $\ell > 4$.
 \end{itemize}
Thus, if $\ell=4$
 \begin{align*}
     \textup{co}_2(G(D'))-\textup{co}_2(G(D))\geq O(1)+\frac{n^4}{(k-1)^2}\left(
     2- 4\left(1-\frac{1}{k-1}\right)^2+2\left(1-\frac{2}{k-1}\right)^2 
     \right)>0,
 \end{align*}
and if $\ell>4$
 \begin{align*}
     \textup{co}_2(G(D'))-\textup{co}_2(G(D))\geq O(1)+\frac{n^4}{(k-1)^2}\left(
     1- 2\left(1-\frac{1}{k-1}\right)^2+\left(1-\frac{2}{k-1}\right)^2 
     \right)>0,
 \end{align*}
a contradiction. Therefore, $D$ contains no cycle of length at least $4$. Next, towards a contradiction, suppose that $D$ contains at least two cycles of length $3$. Let $v_1,v_2,v_3$ and $v_4,v_5,v_6$ be the vertices of two 3-cycles. Let $D'$ be the directed graph constructed from $D$ by replacing those two 3-cycles with three 2-cycles $v_1,v_2$ and $v_3,v_4$ and $v_5,v_6$. Performing a similar analysis to the one above, we get that
 \begin{align*}
     \textup{co}_2(G(D'))-\textup{co}_2(G(D))=O(1)+\frac{n^4}{(k-1)^2}\left(3+3\left(1-\frac{2}{k-1}\right)^2-6\left(1-\frac{1}{k-1}\right)^2 \right)>0,
 \end{align*}
 a contradiction. Thus, we can conclude that $D$ contains at most one 3-cycle. Hence, $D$ is isomorphic to $D^*_k$.
\end{proof}
The directed graph $D_k^*$ contains a 3-cycle iff $k$ is odd.
Based on Lemma~\ref{mubayiconstructions} it seems reasonable to conjecture that in the case when $k$ is odd the hypergraph $G(D_k^*)$ could be an asymptotically extremal example in the $\ell_2$-norm. 
\begin{ques}
Let $k\geq 7$ odd and $\ell=(k-1)/2$. Is
\begin{align*}
\sigma(K_{k}^3)= \lim_{n \to \infty} \frac{\textup{co}_2(G(D^*_k))}{\binom{n}{2} (n-2)^2}= 1-\frac{2}{\ell^2} +\frac{1}{\ell^3} \ ?
\end{align*}
\end{ques}
The situation is slightly different for even $k$. In this case, it is better to consider an unbalanced version of $G(D^*_k)$ with parts of $G(D_k^*)$ corresponding to the unique 3-cycle receiving different weights to the parts involved in 2-cycles. Denote by $G^*(D^*_k)$ the $3$-graph with the largest codegree squared sum among the following 3-graphs $G$. Partition the vertex set of $G$ into $[n]=V_1\cup \ldots \cup V_{k-1}$, where the class sizes are balanced as follow:
\begin{itemize}
    \item $||V_i|-|V_j||\leq 1$ for all $i\neq j$ with $i,j\leq k-4$ and
    \item $||V_i|-|V_j||\leq 1$ for all $i\neq j$ with $ k-3 \leq i,j\leq k-1$.  
\end{itemize}
Again, a triple is a non-edge in $G^*(D^*_k)$ iff it is contained in some $V_i$ or if it has two vertices in $V_i$ and one vertex in $V_j$ where $(i,j)$ is an arc of $D^*_k$.

\begin{ques}
Let $k\geq 6$ even. Is
 \begin{align*}\sigma(K_{k}^3)=\lim_{n \to \infty} \frac{\textup{co}_2(G^*(D_k^*))}{\binom{n}{2} (n-2)^2} \ ?
 \end{align*}
 \end{ques}

\section{Proof of Theorem~\ref{F33exact}}
\label{F33section}
In this section we prove Theorem~\ref{F33exact}, i.e., we determine the codegree squared extremal number of $F_{3,3}$. Flag algebras give us the following corresponding asymptotical result and also a weak stability version.

\begin{lemma}
\label{flag F33} 
For all $\varepsilon>0$ there exists $\delta>0$ and $n_0$ such that for all $n\geq n_0$: if $G$ is an $F_{3,3}$-free 3-uniform graph on $n$ vertices with $\textup{co}_2(G)\geq (1-\delta)\frac{5}{8} n^4/2$, then the densities of all 3-graphs on $4,5$ and $6$ vertices in $G$ that are not contained in $B_n$ are at most $\varepsilon$. Additionally, 
\begin{align*}
    \sigma(F_{3,3})=\frac{5}{8}.
\end{align*}
\end{lemma}
This result implies the following stability theorem.

\begin{theo}
\label{F33stability}
For every $\varepsilon > 0$ there is $\delta>0$ and $n_0$ such that if $G$ is an $F_{3,3}$-free 3-uniform hypergraph on $n \geq n_0$ vertices with $\textup{co}_2(G)\geq (1-\delta) \frac{5}{8} \frac{n^4}{2}$, then we can partition $V(G)$ as $A \cup B$ such that 
$e(A)+e(B)\leq \varepsilon n^3$ and $e(A,B)\geq \frac{1}{8} n^3-\varepsilon n^3$.

\end{theo}
\begin{proof} 
The proof is the same as the proof of Theorem~\ref{K53stability}, except instead of applying Lemma~\ref{K53flag} we apply Lemma~\ref{flag F33}.
\end{proof}
We now determine the exact extremal number by using the stability result, Theorem~\ref{F33stability}, and a standard cleaning technique, see for example \cite{Fanoplane,PikhurkoK43,MubayiF5,FanoFuredi2}. To do so we will first prove the statement under an additional universal minimum-degree-type assumption.

\begin{theo}
\label{F33exactmindeg}
There exists $n_0$ such that for all $n\geq n_0$ the following holds. Let $G$ be an $F_{3,3}$-free $n$-vertex $3$-graph such that 

\begin{align}
\label{minqassumption}
    q(x):= \sum_{y\in V,y\neq x} d(x,y)^2 + 2\sum_{\{v,w\}\in E(L(x))}d(v,w) \geq \frac{5}{4}n^3-6n^2=:d(n)
\end{align}
for all $x\in V(G)$. Then, 
\begin{align*}
 \textup{co}_2(G) \leq \textup{co}_2(B_n)= \binom{\left\lceil\frac{n}{2}\right\rceil}{2} \left\lfloor \frac{n}{2} \right\rfloor^2+\binom{\left\lfloor\frac{n}{2}\right\rfloor}{2} \left\lceil \frac{n}{2} \right\rceil^2+ \left\lceil \frac{n}{2} \right\rceil \left\lfloor\frac{n}{2}\right\rfloor (n-2)^2.
 \end{align*}
Furthermore, $B_n$ is the unique such $3$-graph $G$ satisfying $\textup{co}_2(G)=\textup{exco}_2(n,F_{3,3})$. \end{theo}

\begin{proof}
Let $G$ be a $3$-uniform $F_{3,3}$-free hypergraph which has a codegree squared sum at least $\textup{co}_2(G)\geq \textup{co}_2(B_n)$ and satisfies 
\eqref{minqassumption}. Choose $\varepsilon=10^{-10}$ and apply Theorem~\ref{F33stability}. We get a vertex partition $A \cup B$ with $e(A)+e(B)\leq \varepsilon n^3$ and $e^c(A,B)\leq \varepsilon n^3$. Among all such partitions choose one which minimises $e(A)+e(B)$. 
We can assume that $|L_B(x)|\geq |L_A(x)|$ for all $x\in A$ and $|L_A(x)|\geq |L_B(x)|$ for all $x\in B$, as otherwise we could switch a vertex from one class to the other class and strictly decrease both $e(A)+e(B)$ and $e^c(A,B)$, a contradiction. This is not possible, because we chose $A$ and $B$ minimising $e(A)+e(B)$. We start by making an observation about the class sizes. 

\begin{claim}
\label{classsizes2}
We have
\begin{align*}
\left\lvert |A|-\frac{n}{2} \right\lvert \leq 2\sqrt{\varepsilon}n  \quad \text{and} \quad \left\lvert |B|-\frac{n}{2} \right\lvert \leq 2\sqrt{\varepsilon}n . \quad 
\end{align*}
\end{claim}
\begin{proof}
Assume that $|A|< n/2 - 2\sqrt{\varepsilon}n$. Then, we have 
\begin{align*}
    e(A,B)&\leq \binom{|A|}{2}|B|+|A|\binom{|B|}{2}\leq \frac{1}{2}|A|(n-|A|)n \\
    &<\frac{1}{2}\left( \frac{n}{2}-2\sqrt{\varepsilon}n\right)\left( \frac{n}{2}+2\sqrt{\varepsilon}n\right)n < \frac{1}{8}n^3- \varepsilon n^3,
\end{align*}
a contradiction. Thus, $|A|\geq n/2-2\sqrt{\varepsilon}n$. Similarly, we get $|B|\geq n/2 - 2\sqrt{\varepsilon}n$. 
\end{proof}
Define \emph{junk} sets $J_A,J_B$ to be the sets of vertices which are not typical, i.e.,     
\begin{align*}
    J_A&:= \left\{ x\in A: |L^c_{A,B}(x)|\geq \sqrt{\varepsilon} n^2\right\} \cup \left\{ x\in A: |L_{A}(x)|\geq \sqrt{\varepsilon} n^2\right\}, \text{ and}\\
    J_B&:= \left\{ x\in B: |L^c_{A,B}(x)|\geq \sqrt{\varepsilon} n^2\right\} \cup \left\{ x\in B: |L_{B}(x)|\geq \sqrt{\varepsilon} n^2\right\}.
\end{align*}  
These junk sets need to be small.
\begin{claim}
We have $|J_A|,|J_B|\leq 5\sqrt \varepsilon n$.
\end{claim}
\begin{proof}
 Towards contradiction assume that $|J_A|>5\sqrt{\varepsilon}n$. Then the number of vertices $x\in J_A$ satisfying $|L^c_{A,B}(x)|\geq \sqrt{\varepsilon}n^2$ is at least $2\sqrt{\varepsilon}n$ or the number of vertices $x\in J_A$ satisfying $|L_{A}(x)|\geq \sqrt{\varepsilon}n^2$ is at least $3\sqrt{\varepsilon}n$. If the first case holds, then we get $e^c(A,B)>\varepsilon n^3$. In the second case we have $e(A)>\varepsilon n^3$. Both are in contradiction with the choice of the partition $A \cup B$. Thus, $|J_A|\leq 5\sqrt{\varepsilon}n$. The second statement of this claim, $|J_B|\leq 5\sqrt{\varepsilon}n$, follows by a similar argument. \end{proof}

\begin{claim}
\label{F5cleaning1}
$A\setminus J_A$ and $B\setminus J_B$ are independent sets. 
\end{claim}
\begin{proof}
If there is an edge $a_1a_2a_3$ with $a_1,a_2,a_3\in A\setminus J_A$, since all its vertices satisfy $|L_B^c(a_i)|\leq \sqrt{\varepsilon}n^2$, we can find a triangle in $L_B(a_1)\cap L_B(a_2)\cap L_B(a_3)$, call its vertices $b_1,b_2,b_3$. However, now $ \{b_1,b_2,b_3,a_1,\allowbreak a_2, a_3\}$ spans an $F_{3,3}$ in $G$, a contradiction. A similar proof gives that $B\setminus J_B$ is an independent set. 
\end{proof}
\begin{claim}
\label{typ1doesnot}
There is no edge $a_1a_2a_3$ with $a_1\in J_A$, $a_2,a_3\in A\setminus J_A$ or with $a_1\in J_B$, $a_2,a_3\in B\setminus J_B$.
\end{claim}
\begin{proof}
Let $a_1a_2a_3$ be an edge with $a_1\in J_A$, $a_2,a_3\in A\setminus J_A$. We show that $q(a_1)<d(n)$, to get a contradiction with \eqref{minqassumption}. Let $M_i$, for $i=2,3$, be the set of non-edges in $L_B(a_i)$ and $L_{A,B}(a_i)$. Set $K=L(a_1)-M_2-M_3$. Since $|M_2|,|M_3|\leq 2\sqrt{\varepsilon}n^2$, we have $|E(K)|\geq |L(a_1)|-4\sqrt{\varepsilon}n^2$.  Let 
\begin{align*}
\Delta = \frac{\displaystyle\max_{x\in A \setminus \{a_1,a_2,a_3\}} |N_K(x) \cap B|}{n}
\end{align*}
be the maximum size of a neighbourhood in the graph $K$ in $B$ of a vertex in $A$, scaled by $n$. We have $0\leq \Delta \leq |B|/n \leq 1/2 + \sqrt{\varepsilon}$. Let $z\in A \setminus \{a_1,a_2,a_3\}$ such that $|N_K(z)\cap B)|=\Delta n$.  Observe that $N_K(z) \cap B$ is an independent set in $K$, otherwise if $v,w\in N_K(z) \cap B$ with $vw\in E(K)$, then $\{v,w,z,a_1,a_2,a_3\}$ spans an $F_{3,3}$ in $G$. Now,
\begin{align}
\label{auxillaryaquation}
     \sum_{x\in V\setminus\{a_1\}} d(a_1,x)^2  &= \sum_{x\in V\setminus\{a_1\}}  \deg_{L(a_1)}(x)^2  \leq 16\sqrt{\varepsilon}n^3 + \sum_{x\in V(K)} \deg_{K}(x)^2,    
\end{align}
because for each edge removed from the link graph $L(a_1)$ the degree squared sum can go down by at most $4n$. Now, we bound the sum on the right hand side of \eqref{auxillaryaquation} from above. For $x\in A, \ \deg_k(x)\leq |A|+\Delta n$ and for $x\in N(z)\cap B,\ \deg_k(x)\leq n-\Delta n$. Thus, we get  
\begin{align}
\label{minassumptionstep4}
    & \quad \quad \sum_{x\in V\setminus\{a_1\}} d(a_1,x)^2  \leq 16\sqrt{\varepsilon}n^3+|A|(|A|+ \Delta n)^2+   \Delta n (n-\Delta n)^2+ (|B|-\Delta n)n^2   \nonumber  \\
    &\leq \left(\frac{n}{2} + 2\sqrt{\varepsilon}n\right) \left(\frac{n}{2} + 2\sqrt{\varepsilon}n+ \Delta n\right)^2 + \Delta n (n-\Delta n)^2   
    + \left(\frac{n}{2} +2\sqrt{\varepsilon}n-\Delta n\right)n^2 + 16\sqrt{\varepsilon}n^3 \nonumber \\
    &\leq n^3 \left ( \frac{1}{2} \left( \frac{1}{2} + \Delta \right)^2 + \Delta \left(1-\Delta \right)^2 + \left(\frac{1}{2}-\Delta \right) +25\sqrt{\varepsilon}    \right) 
    = n^3 \left(\frac{5}{8}+\frac{\Delta}{2} -\frac{3}{2}\Delta^2+\Delta^3 +25\sqrt{\varepsilon} \right).
\end{align}
Furthermore, we can give an upper bound for the second summand in $q(a_1)$:
\begin{align}
\label{minassumptionstep1}
    2\sum_{\{x,y\}\in E(L(a_1))}d(x,y) &\leq 8\sqrt{\varepsilon}n^3+2\sum_{\{x,y\}\in E(K)}d(x,y), 
\end{align}
where we used that for each edge removed from $G$, the sum on the left hand side in \eqref{minassumptionstep1} is lowered by at most $n$. Now, we will give an upper bound for the right hand side of \eqref{minassumptionstep1}. 
For edges $xy\in E(K[A])$ not incident to $J_A$ we have $d_G(x,y)\leq |J_A|+|B|$ because by Claim~\ref{F5cleaning1} they have no neighbour in $A\setminus J_A$. Similarly, for edges $xy\in E(K[B])$ not incident to $J_B$ we have $d_G(x,y)\leq |J_B|+|A|.$ For all other edges $xy\in E(K)$, we will use the trivial bound $d_G(x,y)\leq n$. We have 

\begin{align}
\label{minassumptionstep2}
    2\sum_{\{x,y\}\in E(L(a_1))}d(x,y) &\leq 8\sqrt{\varepsilon}n^3 + 2   \Big( e(K[A,B])n + e(K[A])(|J_A|+|B|)+ |J_A||A| n   \nonumber  \\
    &+ e(K[B]) (|J_B|+|B|)+ |J_B||B| n \Big). 
\end{align}
By the choice of our partition we have $|L_A(x_1)|\leq |L_B(x_1)|$ and thus $e(K[A])\leq e(K[B])+4\sqrt{\varepsilon}n^2$. Therefore, by upper bounding the right hand side in \eqref{minassumptionstep2} we get

\begin{align}
\label{minassumptionstep3}
    2\sum_{\{x,y\}\in E(L(a_1))}d(x,y) &\leq 2 \left(\Delta n^2 |A| + 2e(K[B])  \left(7\sqrt{\varepsilon}n+ \frac{n}{2}\right)+ 18\sqrt{\varepsilon}n^3 \right) \nonumber \\
    &\leq 2n^3\left( \frac{\Delta}{2} +  \frac{e(G[B])}{n^2} +30\sqrt{\varepsilon} \right) \nonumber \\
    &\leq 2n^3\left( \frac{\Delta}{2} +  \Delta \left(\frac{|B|}{n}-\Delta\right) + \frac{1}{4}\left(\frac{|B|}{n}-\Delta\right)^2 +30\sqrt{\varepsilon} \right) \nonumber \\
    &\leq 2n^3\left( \frac{\Delta}{2} +  \Delta \left(\frac{1}{2}-\Delta\right) + \frac{1}{4}\left(\frac{1}{2}-\Delta\right)^2 +40\sqrt{\varepsilon} \right) \nonumber \\
    &\leq n^3\left( -\frac{3}{2}\Delta^2 + \frac{3}{2}\Delta + \frac{1}{8}+ 80 \sqrt{\varepsilon}\right),
\end{align}
   where we used that $e(K[B])\leq \Delta n (|B|-\Delta n)+ \frac{(|B|-\Delta n)^2}{4}$, because $K[B]$ contains an independent set of size $\Delta n$ and is triangle-free. Now, we can combine \eqref{minassumptionstep4} and \eqref{minassumptionstep3} to upper bound $q(a_1)$.
\begin{align*}
   q(a_1) &\leq n^3 \left(\frac{5}{8}+\frac{\Delta}{2} -\frac{3}{2}\Delta^2+\Delta^3 +25\sqrt{\varepsilon} \right) + n^3\left( -\frac{3}{2}\Delta^2 + \frac{3}{2}\Delta + \frac{1}{8}+ 80 \sqrt{\varepsilon}\right) \\
   &= n^3 \left(\Delta^3-3\Delta^2+ 2\Delta +\frac{3}{4}+ 105\sqrt{\varepsilon} \right)\leq \left(\frac{2}{3\sqrt{3}}+\frac{3}{4}+105\sqrt{\varepsilon} \right) n^3 < \frac{5}{4}n^3-6n^2,
\end{align*}
contradicting \eqref{minqassumption}. In the second-to-last inequality we used that the polynomial $\Delta^3-3\Delta^2+ 2\Delta$ has its maximum in $[0,1]$ at $\Delta=1-\frac{1}{\sqrt{3}}$.
   \end{proof}
Now, we can make use of Claim~\ref{typ1doesnot} to show that there is no edge inside $A$, respectively inside $B$.   
\begin{claim}
\label{indset}
$A$ and $B$ are independent sets.
\end{claim}
\begin{proof}
Let $\{a_1,a_2,a_3\} \subset A$ span an edge. Again, $L_B(a_1) \cap L_B(a_2) \cap L_B(a_3)$ is triangle-free. Thus, $|L_B(a_1) \cap L_B(a_2) \cap L_B(a_3)|\leq |B|^2 /4$. By the pigeonhole principle, we may assume without loss of generality that $|L_B(a_1)|\leq 5|B|^2/12.$ Furthermore, by Claims~\ref{F5cleaning1} and \ref{typ1doesnot}, $|L_A(a_1)|\leq |J_A||A|\leq 5\sqrt{\varepsilon}n^2$. Again, our strategy will be to give an upper bound on $q(a_1)$. Let $L$ be the graph obtained from $L(a_1)$ by removing all edges inside $A$.
\begin{align}
   \sum_{x\in V\setminus\{a_1\}} d(a_1,x)^2  &= \sum_{x\in V\setminus\{a_1\}}  \deg_{L(a_1)}(x)^2  \leq 20\sqrt{\varepsilon }n^3+ \sum_{x\in V(L)} \deg_{L}(x)^2 \nonumber \\
    &\leq 20\sqrt{\varepsilon }n^3+ |B| n^2 + |A||B|^2 \leq n^3 \left( \frac{5}{8}+ 30\sqrt{\varepsilon}\right). \label{part1equal}
\end{align}
Furthermore, 
\begin{align}
    2\sum_{\{x,y\}\in E(L(a_1))}d(x,y) &\leq 10\sqrt{\varepsilon}n^3+ 2\sum_{xy\in E(L)}d(x,y) \nonumber \\
    &\leq 2\left(\frac{5}{12}|B|^2 \left(|A|+|J_B|\right) + 5\sqrt{\varepsilon}n^3 + |A||B|n \right) \nonumber \\
    &\leq 2n^3\left(\frac{5}{96}  + 20\sqrt{\varepsilon} + \frac{1}{4} \right)=n^3 \left(\frac{29}{48}+ 40 \sqrt{\varepsilon}\right). \label{part2equal}
\end{align}
Thus, by combining \eqref{part1equal} and \eqref{part2equal}, we give an upper bound on $q(a_1)$,
\begin{align*}
q(a_1)\leq \left( \frac{5}{8}+ 30\sqrt{\varepsilon}\right)n^3 + n^3 \left(\frac{29}{48}+ 40 \sqrt{\varepsilon}\right)=n^3\left(\frac{59}{48}+ 70\sqrt{\varepsilon} \right)< \frac{5}{4}n^3-6n^2,
\end{align*}
contradicting \eqref{minqassumption}. Therefore $A$ is an independent set. By a similar argument $B$ is also an independent set. 

\end{proof}
By Claim~\ref{indset}, $G$ is $2$-colourable. Since among all $2$-colourable $3$-graphs $B_n$ has the largest codegree squared sum, we conclude $\textup{co}_2(G)\leq \textup{co}_2(B_n)$. This completes the proof of Theorem~\ref{F33exactmindeg}.
  
\end{proof}

We now complete the proof of Theorem~\ref{F33exactmindeg} by showing that imposing the additional assumption \eqref{minqassumption} is not more restrictive.  

\begin{proof}[Proof of Theorem~\ref{F33exact}]
Let $G$ be an $n$-vertex $3$-uniform $F_{3,3}$-free hypergraph which has a codegree squared sum at least $\textup{co}_2(G)\geq \textup{co}_2(B_n)$. Set $d(n)=5/4 n^3-6n^2$ and note that $\textup{co}_2(B_n)-\textup{co}_2(B_{n-1})>d(n)+1$. We claim that we can assume that every vertex $x\in V(G)$ sa\-tis\-fies \eqref{minqassumption}.
Otherwise, we can remove a vertex $x$ with $q(x)<d(n)$ to get $G_{n-1}$ with $\textup{co}_2(G_{n-1})\geq \textup{co}_2(B_n)-d(n)\geq \textup{co}_2(B_{n-1})+1.$ By repeating this process as long as possible, we obtain a sequence of hypergraphs $G_m$ on $m$ vertices with $\textup{co}_2(G_m)\geq \textup{co}_2(B_m)+n-m$, where $G_m$ is the hypergraph obtained from $G_{m+1}$ by deleting a vertex $x$ with $q(x) \leq d(m+1)$. We cannot continue until we reach a hypergraph on $n_0=n^{1/4}$ vertices, as then $\textup{co}_2(G_{n_0})>n-n_0 > \binom{n_0}{2}(n_0-2)^2$ which is impossible. Therefore, the process stops at some $n'$ where $n\geq n'\geq n_0$ and we obtain the corresponding hypergraph $G_{n'}$ satisfying $q(x)\geq d(n')$ for all $x\in V(G_{n'})$ and $\textup{co}_2(G_{n'})\geq \textup{co}_2(B_{n'})$ (with strict inequality if $n>n'$). Hence, we can assume that $G$ satisfies $q(x)\geq d(n')$ for all $x\in V(G_{n'})$. Applying Theorem~\ref{F33exactmindeg} finishes the proof. 
\end{proof}

\section*{Acknowledgements}
We thank an anonymous referee for many useful comments and suggestions, in particular for pointing out a shorter proof of Theorem~\ref{K43stability}.

\bibliographystyle{abbrvurl}
\bibliography{codegreesquared}

\begin{thebibliography}{10}

\bibitem{FAPhylogenetic}
N.~Alon, H.~Naves, and B.~Sudakov.
\newblock On the maximum quartet distance between phylogenetic trees.
\newblock {\em SIAM J. Discrete Math.}, 30(2):718--735, 2016.
\newblock \href {http://dx.doi.org/10.1137/15M1041754}
  {\path{doi:10.1137/15M1041754}}.

\bibitem{Baber}
R.~Baber.
\newblock Tur\'an densities of hypercubes.
\newblock {\em arXiv preprint}, 2012.
\newblock \href {http://arxiv.org/abs/1201.3587} {\path{arXiv:1201.3587}}.

\bibitem{BaberTalbot}
R.~Baber and J.~Talbot.
\newblock Hypergraphs do jump.
\newblock {\em Combin. Probab. Comput.}, 20(2):161--171, 2011.
\newblock \href {http://dx.doi.org/10.1017/S0963548310000222}
  {\path{doi:10.1017/S0963548310000222}}.

\bibitem{BalCleLid}
J.~Balogh, F.~C. Clemen, and B.~Lidick\'{y}.
\newblock Hypergraph {T}ur\'an problems in $\ell_2$-norm.
\newblock 2021.
\newblock To appear in \textit{BCC}.
\newblock \href {http://arxiv.org/abs/2108.10406} {\path{arXiv:2108.10406}}.

\bibitem{FAColor2}
J.~Balogh, P.~Hu, B.~Lidick\'{y}, F.~Pfender, J.~Volec, and M.~Young.
\newblock Rainbow triangles in three-colored graphs.
\newblock {\em J. Combin. Theory Ser. B}, 126:83--113, 2017.
\newblock \href {http://dx.doi.org/10.1016/j.jctb.2017.04.002}
  {\path{doi:10.1016/j.jctb.2017.04.002}}.

\bibitem{PAPerm1}
J.~Balogh, P.~Hu, B.~Lidick\'{y}, O.~Pikhurko, B.~Udvari, and J.~Volec.
\newblock Minimum number of monotone subsequences of length 4 in permutations.
\newblock {\em Combin. Probab. Comput.}, 24(4):658--679, 2015.
\newblock \href {http://dx.doi.org/10.1017/S0963548314000820}
  {\path{doi:10.1017/S0963548314000820}}.

\bibitem{FAGeometry}
J.~Balogh, B.~Lidick\'{y}, and G.~Salazar.
\newblock Closing in on {H}ill's conjecture.
\newblock {\em SIAM J. Discrete Math.}, 33(3):1261--1276, 2019.
\newblock \href {http://dx.doi.org/10.1137/17M1158859}
  {\path{doi:10.1137/17M1158859}}.

\bibitem{Bollobascancellative}
B.~Bollob\'{a}s.
\newblock Three-graphs without two triples whose symmetric difference is
  contained in a third.
\newblock {\em Discrete Math.}, 8:21--24, 1974.
\newblock \href {http://dx.doi.org/10.1016/0012-365X(74)90105-8}
  {\path{doi:10.1016/0012-365X(74)90105-8}}.

\bibitem{csdp}
B.~Borchers.
\newblock C{SDP}, a {C} library for semidefinite programming.
\newblock volume 11/12, pages 613--623. 1999.
\newblock Interior point methods.
\newblock \href {http://dx.doi.org/10.1080/10556789908805765}
  {\path{doi:10.1080/10556789908805765}}.

\bibitem{K43brown}
W.~G. Brown.
\newblock On an open problem of {P}aul {T}ur\'{a}n concerning {$3$}-graphs.
\newblock In {\em Studies in pure mathematics}, pages 91--93. Birkh\"{a}user,
  Basel, 1983.

\bibitem{ChungLutetrahedron}
F.~Chung and L.~Lu.
\newblock An upper bound for the {T}ur\'{a}n number {$t_3(n,4)$}.
\newblock {\em J. Combin. Theory Ser. A}, 87(2):381--389, 1999.
\newblock \href {http://dx.doi.org/10.1006/jcta.1998.2961}
  {\path{doi:10.1006/jcta.1998.2961}}.

\bibitem{FAcolor}
J.~Cummings, D.~Kr\'{a}l', F.~Pfender, K.~Sperfeld, A.~Treglown, and M.~Young.
\newblock Monochromatic triangles in three-coloured graphs.
\newblock {\em J. Combin. Theory Ser. B}, 103(4):489--503, 2013.
\newblock \href {http://dx.doi.org/10.1016/j.jctb.2013.05.002}
  {\path{doi:10.1016/j.jctb.2013.05.002}}.

\bibitem{TetrahedronCaen}
D.~de~Caen.
\newblock On upper bounds for {$3$}-graphs without tetrahedra.
\newblock {\em Congr. Numer.}, 62:193--202, 1988.
\newblock Seventeenth Manitoba Conference on Numerical Mathematics and
  Computing (Winnipeg, MB, 1987).

\bibitem{Erdos3partite}
P.~Erd\H{o}s.
\newblock On extremal problems of graphs and generalized graphs.
\newblock {\em Israel J. of Math.}, 2(3):183--190, 1964.

\bibitem{ErdosSimonovits}
P.~Erd\H{o}s and M.~Simonovits.
\newblock A limit theorem in graph theory.
\newblock {\em Studia Sci. Math. Hungar.}, 1:51--57, 1966.

\bibitem{Supersaterdos}
P.~Erd\H{o}s and M.~Simonovits.
\newblock Supersaturated graphs and hypergraphs.
\newblock {\em Combinatorica}, 3(2):181--192, 1983.
\newblock \href {http://dx.doi.org/10.1007/BF02579292}
  {\path{doi:10.1007/BF02579292}}.

\bibitem{ErdosStone}
P.~Erd\H{o}s and A.~H. Stone.
\newblock On the structure of linear graphs.
\newblock {\em Bull. Amer. Math. Soc.}, 52:1087--1091, 1946.
\newblock \href {http://dx.doi.org/10.1090/S0002-9904-1946-08715-7}
  {\path{doi:10.1090/S0002-9904-1946-08715-7}}.

\bibitem{codegreeFalgas}
V.~Falgas-Ravry.
\newblock On the codegree density of complete 3-graphs and related problems.
\newblock {\em Electron. J. Combin.}, 20(4):Paper 28, 14, 2013.

\bibitem{codF32falgas}
V.~Falgas-Ravry, E.~Marchant, O.~Pikhurko, and E.~R. Vaughan.
\newblock The codegree threshold for 3-graphs with independent neighborhoods.
\newblock {\em SIAM J. Discrete Math.}, 29(3):1504--1539, 2015.
\newblock \href {http://dx.doi.org/10.1137/130926997}
  {\path{doi:10.1137/130926997}}.

\bibitem{FalgasK4-}
V.~Falgas-Ravry, O.~Pikhurko, E.~Vaughan, and J.~Volec.
\newblock The codegree threshold of ${K}_4^-$.
\newblock {\em Electronic Notes in Discrete Mathematics}, 61:407--413, 2017.

\bibitem{MR2988862}
V.~Falgas-Ravry and E.~R. Vaughan.
\newblock Tur\'{a}n {$H$}-densities for 3-graphs.
\newblock {\em Electron. J. Combin.}, 19(3):Paper 40, 26, 2012.
\newblock \href {http://dx.doi.org/10.37236/2733} {\path{doi:10.37236/2733}}.

\bibitem{RavryTuran}
V.~Falgas-Ravry and E.~R. Vaughan.
\newblock Applications of the semi-definite method to the {T}ur\'{a}n density
  problem for 3-graphs.
\newblock {\em Combin. Probab. Comput.}, 22(1):21--54, 2013.
\newblock \href {http://dx.doi.org/10.1017/S0963548312000508}
  {\path{doi:10.1017/S0963548312000508}}.

\bibitem{K43Fonderflaass}
D.~G. Fon-Der-Flaass.
\newblock A method for constructing {$(3,4)$}-graphs.
\newblock {\em Mat. Zametki}, 44(4):546--550, 559, 1988.
\newblock \href {http://dx.doi.org/10.1007/BF01158925}
  {\path{doi:10.1007/BF01158925}}.

\bibitem{F5Frankl}
P.~Frankl and Z.~F\"{u}redi.
\newblock A new generalization of the {E}rd{\H o}s-{K}o-{R}ado theorem.
\newblock {\em Combinatorica}, 3(3--4):341--349, 1983.
\newblock \href {http://dx.doi.org/10.1007/BF02579190}
  {\path{doi:10.1007/BF02579190}}.

\bibitem{MR2465761}
A.~Frohmader.
\newblock More constructions for {T}ur\'{a}n's {$(3,4)$}-conjecture.
\newblock {\em Electron. J. Combin.}, 15(1):Research Paper 137, 23, 2008.
\newblock \href {http://dx.doi.org/10.37236/861} {\path{doi:10.37236/861}}.

\bibitem{FanoFuredi2}
Z.~F\"{u}redi and M.~Simonovits.
\newblock Triple systems not containing a {F}ano configuration.
\newblock {\em Combin. Probab. Comput.}, 14(4):467--484, 2005.
\newblock \href {http://dx.doi.org/10.1017/S0963548305006784}
  {\path{doi:10.1017/S0963548305006784}}.

\bibitem{MR3048207}
J.~Goldwasser and R.~Hansen.
\newblock The exact {T}ur\'{a}n number of {$F(3,3)$} and all extremal
  configurations.
\newblock {\em SIAM J. Discrete Math.}, 27(2):910--917, 2013.
\newblock \href {http://dx.doi.org/10.1137/110841837}
  {\path{doi:10.1137/110841837}}.

\bibitem{FAgraphs}
A.~Grzesik, P.~Hu, and J.~Volec.
\newblock Minimum number of edges that occur in odd cycles.
\newblock {\em J.~Combin. Theory Ser. B}, 137:65--103, 2019.
\newblock \href {http://dx.doi.org/10.1016/j.jctb.2018.12.003}
  {\path{doi:10.1016/j.jctb.2018.12.003}}.

\bibitem{FAdigraphs}
J.~Hladk\'{y}, D.~Kr\'{a}l', and S.~Norin.
\newblock Counting flags in triangle-free digraphs.
\newblock {\em Combinatorica}, 37(1):49--76, 2017.
\newblock \href {http://dx.doi.org/10.1007/s00493-015-2662-5}
  {\path{doi:10.1007/s00493-015-2662-5}}.

\bibitem{Jansonhypergeometric}
S.~Janson, T.~{\L}uczak, and A.~Rucinski.
\newblock {\em Random graphs}, volume~45.
\newblock John Wiley \& Sons, 2011.

\bibitem{Keevashsurvey}
P.~Keevash.
\newblock Hypergraph {T}ur\'{a}n problems.
\newblock In {\em Surveys in combinatorics 2011}, volume 392 of {\em London
  Math. Soc. Lecture Note Ser.}, pages 83--139. Cambridge Univ. Press,
  Cambridge, 2011.

\bibitem{MubayiF5}
P.~Keevash and D.~Mubayi.
\newblock Stability theorems for cancellative hypergraphs.
\newblock {\em J. Combin. Theory Ser. B}, 92(1):163--175, 2004.
\newblock \href {http://dx.doi.org/10.1016/j.jctb.2004.05.003}
  {\path{doi:10.1016/j.jctb.2004.05.003}}.

\bibitem{MR2912791}
P.~Keevash and D.~Mubayi.
\newblock The {T}ur\'{a}n number of {$F_{3,3}$}.
\newblock {\em Combin. Probab. Comput.}, 21(3):451--456, 2012.
\newblock \href {http://dx.doi.org/10.1017/S0963548311000678}
  {\path{doi:10.1017/S0963548311000678}}.

\bibitem{Fanoplane}
P.~Keevash and B.~Sudakov.
\newblock The {T}ur\'{a}n number of the {F}ano plane.
\newblock {\em Combinatorica}, 25(5):561--574, 2005.
\newblock \href {http://dx.doi.org/10.1007/s00493-005-0034-2}
  {\path{doi:10.1007/s00493-005-0034-2}}.

\bibitem{K43Kostochka}
A.~V. Kostochka.
\newblock A class of constructions for {T}ur\'{a}n's {$(3,\,4)$}-problem.
\newblock {\em Combinatorica}, 2(2):187--192, 1982.
\newblock \href {http://dx.doi.org/10.1007/BF02579317}
  {\path{doi:10.1007/BF02579317}}.

\bibitem{FAGeom2}
D.~Kr\'{a}l', L.~Mach, and J.-S. Sereni.
\newblock A new lower bound based on {G}romov's method of selecting heavily
  covered points.
\newblock {\em Discrete Comput. Geom.}, 48(2):487--498, 2012.
\newblock \href {http://dx.doi.org/10.1007/s00454-012-9419-3}
  {\path{doi:10.1007/s00454-012-9419-3}}.

\bibitem{FATrounaments}
N.~Linial and A.~Morgenstern.
\newblock On the number of 4-cycles in a tournament.
\newblock {\em J. Graph Theory}, 83(3):266--276, 2016.
\newblock \href {http://dx.doi.org/10.1002/jgt.21996}
  {\path{doi:10.1002/jgt.21996}}.

\bibitem{LoMark}
A.~Lo and K.~Markstr\"{o}m.
\newblock {$\ell$}-degree {T}ur\'{a}n density.
\newblock {\em SIAM J. Discrete Math.}, 28(3):1214--1225, 2014.
\newblock \href {http://dx.doi.org/10.1137/120895974}
  {\path{doi:10.1137/120895974}}.

\bibitem{AllanZhao}
A.~Lo and Y.~Zhao.
\newblock Codegree {T}ur\'{a}n density of complete {$r$}-uniform hypergraphs.
\newblock {\em SIAM J. Discrete Math.}, 32(2):1154--1158, 2018.
\newblock \href {http://dx.doi.org/10.1137/18M1163956}
  {\path{doi:10.1137/18M1163956}}.

\bibitem{MubayiFano}
D.~Mubayi.
\newblock The co-degree density of the {F}ano plane.
\newblock {\em J. Combin. Theory Ser. B}, 95(2):333--337, 2005.
\newblock \href {http://dx.doi.org/10.1016/j.jctb.2005.06.001}
  {\path{doi:10.1016/j.jctb.2005.06.001}}.

\bibitem{Mubayicode}
D.~Mubayi and Y.~Zhao.
\newblock Co-degree density of hypergraphs.
\newblock {\em J. Combin. Theory Ser. A}, 114(6):1118--1132, 2007.
\newblock \href {http://dx.doi.org/10.1016/j.jcta.2006.11.006}
  {\path{doi:10.1016/j.jcta.2006.11.006}}.

\bibitem{codegreeconj}
B.~Nagle.
\newblock Tur\'{a}n related problems for hypergraphs.
\newblock In {\em Proceedings of the {T}hirtieth {S}outheastern {I}nternational
  {C}onference on {C}ombinatorics, {G}raph {T}heory, and {C}omputing ({B}oca
  {R}aton, {FL}, 1999)}, volume 136, pages 119--127, 1999.

\bibitem{PikhurkoK43}
O.~Pikhurko.
\newblock The minimum size of 3-graphs without a 4-set spanning no or exactly
  three edges.
\newblock {\em European J. Combin.}, 32(7):1142--1155, 2011.
\newblock \href {http://dx.doi.org/10.1016/j.ejc.2011.03.006}
  {\path{doi:10.1016/j.ejc.2011.03.006}}.

\bibitem{FAgraphs2}
O.~Pikhurko, J.~Slia\v{c}an, and K.~Tyros.
\newblock Strong forms of stability from flag algebra calculations.
\newblock {\em J. Combin. Theory Ser. B}, 135:129--178, 2019.
\newblock \href {http://dx.doi.org/10.1016/j.jctb.2018.08.001}
  {\path{doi:10.1016/j.jctb.2018.08.001}}.

\bibitem{flagsRaz}
A.~A. Razborov.
\newblock Flag algebras.
\newblock {\em J. Symbolic Logic}, 72(4):1239--1282, 2007.
\newblock \href {http://dx.doi.org/10.2178/jsl/1203350785}
  {\path{doi:10.2178/jsl/1203350785}}.

\bibitem{RazbarovK43}
A.~A. Razborov.
\newblock On 3-hypergraphs with forbidden 4-vertex configurations.
\newblock {\em SIAM J. Discrete Math.}, 24(3):946--963, 2010.
\newblock \href {http://dx.doi.org/10.1137/090747476}
  {\path{doi:10.1137/090747476}}.

\bibitem{MR3186665}
A.~A. Razborov.
\newblock Flag algebras: an interim report.
\newblock In {\em The mathematics of {P}aul {E}rd\H{o}s. {II}}, pages 207--232.
  Springer, New York, 2013.
\newblock \href {http://dx.doi.org/10.1007/978-1-4614-7254-4\_16}
  {\path{doi:10.1007/978-1-4614-7254-4\_16}}.

\bibitem{MR3135939}
A.~A. Razborov.
\newblock What is{$\ldots$}a flag algebra?
\newblock {\em Notices Amer. Math. Soc.}, 60(10):1324--1327, 2013.
\newblock \href {http://dx.doi.org/10.1090/noti1051}
  {\path{doi:10.1090/noti1051}}.

\bibitem{MR3548293}
C.~Reiher, V.~R\"{o}dl, and M.~Schacht.
\newblock Embedding tetrahedra into quasirandom hypergraphs.
\newblock {\em J. Combin. Theory Ser. B}, 121:229--247, 2016.
\newblock \href {http://dx.doi.org/10.1016/j.jctb.2016.06.008}
  {\path{doi:10.1016/j.jctb.2016.06.008}}.

\bibitem{MR3764068}
C.~Reiher, V.~R\"{o}dl, and M.~Schacht.
\newblock Hypergraphs with vanishing {T}ur\'{a}n density in uniformly dense
  hypergraphs.
\newblock {\em J. Lond. Math. Soc. (2)}, 97(1):77--97, 2018.
\newblock \href {http://dx.doi.org/10.1112/jlms.12095}
  {\path{doi:10.1112/jlms.12095}}.

\bibitem{indremoval}
V.~R\"{o}dl and M.~Schacht.
\newblock Generalizations of the removal lemma.
\newblock {\em Combinatorica}, 29(4):467--501, 2009.
\newblock \href {http://dx.doi.org/10.1007/s00493-009-2320-x}
  {\path{doi:10.1007/s00493-009-2320-x}}.

\bibitem{Sidorenko1981SystemsOS}
A.~Sidorenko.
\newblock Systems of sets that have the {T}-property.
\newblock {\em Moscow University Mathematics Bulletin 36}, 36:22--26, 1981.

\bibitem{MR1341481}
A.~Sidorenko.
\newblock What we know and what we do not know about {T}ur\'{a}n numbers.
\newblock {\em Graphs Combin.}, 11(2):179--199, 1995.
\newblock \href {http://dx.doi.org/10.1007/BF01929486}
  {\path{doi:10.1007/BF01929486}}.

\bibitem{Sidorenkocode}
A.~Sidorenko.
\newblock Extremal problems on the hypercube and the codegree {T}ur\'{a}n
  density of complete {$r$}-graphs.
\newblock {\em SIAM J. Discrete Math.}, 32(4):2667--2674, 2018.
\newblock \href {http://dx.doi.org/10.1137/17M1151171}
  {\path{doi:10.1137/17M1151171}}.

\bibitem{PAPerm2}
J.~Slia\v{c}an and W.~Stromquist.
\newblock Improving bounds on packing densities of 4-point permutations.
\newblock {\em Discrete Math. Theor. Comput. Sci.}, 19(2):Paper No. 3, 18,
  2017.
\newblock \href {http://dx.doi.org/10.1109/mcse.2017.21}
  {\path{doi:10.1109/mcse.2017.21}}.

\bibitem{sagemath}
{The Sage Developers}.
\newblock {\em {S}ageMath, the {S}age {M}athematics {S}oftware {S}ystem
  ({V}ersion 9)}, 2021.
\newblock URL: \url{https://www.sagemath.org}.

\end{thebibliography}

\end{document}